\documentclass[a4paper,reqno,11pt]{amsart}
\usepackage[a4paper,left=2cm,right=2cm,top=2.5cm,bottom=2.5cm]{geometry}
\usepackage[onehalfspacing]{setspace}
\usepackage[colorlinks=true, linkcolor=magenta, citecolor=cyan]{hyperref}
\usepackage{fontspec}
\usepackage{amsmath}
\usepackage{amsthm}
\usepackage{amssymb}
\usepackage{amsfonts}
\usepackage{mathtools}
\usepackage{colonequals}
\usepackage{dsfont}
\usepackage{enumitem}
\usepackage{mathdots}
\usepackage{fancyhdr}
\usepackage{mdframed}
\usepackage{csquotes}
\usepackage{bm}
\usepackage[normalem]{ulem}

\usepackage{tikz-cd}
\tikzset{
  symbol/.style={
    draw=none,
    every to/.append style={
      edge node={node [sloped, allow upside down, auto=false]{$#1$}}}
  }
}

\uccode`ß="1E9E

\newtheorem{introtheorem}{Theorem}

\newtheorem{theorem}{Theorem}[section]
\newtheorem{proposition}[theorem]{Proposition}
\newtheorem{lemma}[theorem]{Lemma}
\newtheorem{corollary}[theorem]{Corollary}

\theoremstyle{definition}
\newtheorem{introdefinition}[introtheorem]{Definition}
\newtheorem{definition}[theorem]{Definition}
\newtheorem{notation}[theorem]{Notation}
\newtheorem{recollection}[theorem]{Recollection}

\theoremstyle{remark}
\newtheorem{remark}[theorem]{Remark}
\newtheorem{example}[theorem]{Example}

\usepackage[backend=biber,maxalphanames=5,maxbibnames=5,style=alphabetic]{biblatex}
\AtBeginBibliography{\small}

\newcommand{\br}[1]{\left(#1\right)}
\newcommand{\sqb}[1]{\left\lbrack #1\right\rbrack}
\newcommand{\abs}[1]{\left\lvert #1\right\rvert}
\newcommand{\iso}{\stackrel{\sim}{\rightarrow}}

\newcommand{\N}{\mathbb{N}}

\newcommand{\B}{\mathcal{B}}
\newcommand{\C}{\mathcal{C}}
\newcommand{\D}{\mathcal{D}}
\newcommand{\E}{\mathcal{E}}
\newcommand{\F}{\mathcal{F}}

\newcommand{\I}{\mathcal{I}}
\newcommand{\J}{\mathcal{J}}

\newcommand{\T}{\mathcal{T}}

\newcommand{\const}{\mathrm{const}}

\DeclareMathOperator*{\colim}{colim}

\newcommand{\Map}{\operatorname{Map}}
\newcommand{\Hom}{\operatorname{Hom}}

\newcommand{\id}{\operatorname{id}}
\newcommand{\ev}{\operatorname{ev}}
\newcommand{\Fun}{\operatorname{Fun}}

\newcommand{\Un}{\mathrm{Un}}

\newcommand{\Set}{\mathrm{Set}}
\newcommand{\Cat}{\mathrm{Cat}}
\newcommand{\An}{\mathrm{An}}

\newcommand{\Seg}{\mathrm{Seg}}
\newcommand{\PSh}{\mathrm{PSh}}
\newcommand{\Ar}{\mathrm{Ar}}

\newcommand{\lax}{\mathrm{lax}}
\newcommand{\oplax}{\mathrm{oplax}}
\newcommand{\op}{\mathrm{op}}

\newcommand{\singlequote}[1]{\textquotesingle #1\textquotesingle}

\newcommand{\pb}{\arrow[dr,phantom,very near start,"\lrcorner"]}

\newcommand{\bfCat}{\mathbf{Cat}}
\newcommand{\bbCat}{\mathbb{C}\mathrm{at}}
\newcommand{\Ambi}{\mathrm{Ambi}}
\newcommand{\bfAmbi}{\mathbf{Ambi}}
\newcommand{\bbAmbi}{\mathbb{A}\mathrm{mbi}}
\newcommand{\Ortho}{\mathrm{Ortho}}

\newcommand{\Fact}{\mathrm{Fact}}
\newcommand{\bfFact}{\mathbf{Fact}}
\newcommand{\bbFact}{\mathbb{F}\mathrm{act}}
\newcommand{\bfFun}{\mathbf{Fun}}
\newcommand{\bbFun}{\mathbb{F}\mathrm{un}}
\newcommand{\DCat}{\mathrm{DCat}}
\newcommand{\bfCocart}{\mathbf{Cocart}}

\newcommand{\eg}{\mathrm{eg}}
\newcommand{\ing}{\mathrm{in}}
\newcommand{\Sq}{\mathrm{Sq}}
\newcommand{\eq}{\mathrm{eq}}
\newcommand{\Ab}{\mathrm{Ab}}
\newcommand{\Gpd}{\mathrm{Gpd}}
\newcommand{\Hor}{\mathrm{Hor}}
\newcommand{\Ver}{\mathrm{Ver}}
\newcommand{\uple}{\mathrm{uple}}

\newcommand{\bbD}{\mathbb{D}}

\author{Thorger Gei\ss}
\title{Factorization Systems on $\infty$-Categories: Un/straightening and Monadicity}
\address{FB Mathematik und Informatik, Universität Münster, Einsteinstraße 62, 48149 Münster, Germany}
\email{tgeiss@uni-muenster.de}

\begin{document}

\begin{abstract}
We prove two structural results about factorization systems on $\infty$-categories. Firstly, we classify factorization systems on the total spaces of a class of fibrations of $\infty$-categories in terms of factorization systems on the base and the fibers. This will be interpreted as an un/straightening equivalence for $\infty$-categories equipped with a factorization system, using Juran's double $\infty$-categorical framework. Along the way, we develop the general theory of lifting through faithful functors of $n$-uple $\infty$-categories. Secondly, we prove that the forgetful functor from $\infty$-categories equipped with a factorization system to $\infty$-categories is a monadic right adjoint.
\end{abstract}

\maketitle

\tableofcontents

\section{Introduction}

A factorization system on a category $\C$ consists of two classes of morphisms in $\C$—abusing notation due to Barwick, we call these \textit{egressive} and \textit{ingressive} morphisms, denoted by $\twoheadrightarrow$ and $\rightarrowtail$, respectively—such that every morphism in $\C$ factors as an egressive morphism followed by an ingressive morphism and the two classes are \textit{orthogonal}, i.e. every lifting problem
\begin{equation*}
\begin{tikzcd}
a\arrow[d,twoheadrightarrow]\arrow[r] & x\arrow[d,rightarrowtail]\\
b\arrow[r]\arrow[ur,dashed] & y
\end{tikzcd}
\end{equation*}
admits a unique solution. In presence of the former condition, this is equivalent to the egressive-ingressive factorizations being unique up to unique isomorphism.

Factorization systems are classical tools in category theory. In this paper, we study the \textit{homotopy-coherent} generalization of factorization systems, i.e. factorization systems on $\infty$-categories. These were already introduced in the foundational work on $\infty$-categories of Joyal and Lurie (\cite{Joyal}, \cite{HTT}) and enjoy a wide range of applications spanning $\infty$-topos theory (see e.g. \cite{ABFJ}), algebraic patterns \& operads (see e.g. \cite{ChuHaugseng}, \cite{HA}) or the general theory of $(\infty,n)$-categories (see e.g. \cite{Soergel}, \cite{ORW}).

The first question we study in this paper is the \singlequote{(dis-)integration} of factorization systems: given some sort of \singlequote{fibration} $p\colon\E\rightarrow\C$, is it possible to characterize factorization systems on the total space $\E$ in terms of factorization systems on the base $\C$ and on the fibers $\E_c,\,c\in\C$, assuming compatibility with $p$? To this end, we introduce the following class of fibrations of $\infty$-categories equipped with a factorization system.

\begin{introdefinition}[Definition \ref{AmbiFibDefn}]
An \textit{ambifibration} is a functor $p\colon\E^{\dagger}\rightarrow\C^{\dagger}$ of $\infty$-categories equipped with a factorization system (i.e. it preserves both egressive and ingressive morphisms), which admits cocartesian resp. cartesian lifts of egressive resp. ingressive morphisms, which are again egressive resp. ingressive.
\end{introdefinition}

In this context, we obtain the following theorem answering the previous question. It is proven by an explicit construction, which unifies some constructions that are spread throughout the literature.

\begin{introtheorem}[Theorem \ref{FactorizationSystemNaturality}]\label{FactorizationSystemIntro}
Let $\C^{\dagger}$ be an $\infty$-category equipped with a factorization system and $p\colon\E\rightarrow\C$ a functor of $\infty$-categories, which admits cocartesian resp. cartesian lifts of egressive resp. ingressive morphisms. Then the space of ambifibrations $p^{\dagger}\colon\E^{\dagger}\rightarrow\C^{\dagger}$ lifting $p$ is discrete and equivalent to the set of collections of factorization systems on the fibers $\E_c,\,c\in\C$ such that cocartesian resp. cartesian transport along egressive resp. ingressive morphisms preserves egressive resp. ingressive morphisms in the fibers.
\end{introtheorem}

Our next goal is to interpret this as an un/straightening equivalence for $\infty$-categories equipped with a factorization system. To state this, we will employ the double $\infty$-categorical framework of Juran \cite{Juran}. Recall that a \textit{double $\infty$-category} is a homotopy-coherent categorical structure that comes equipped with two types of morphisms, namely \textit{horizontal} and \textit{vertical} morphisms that have respective composition laws, as well as \textit{$2$-cells} or \textit{squares}, depicted as
\begin{equation*}
\begin{tikzcd}
a\arrow[r,""{name=U,below}]\arrow[d] & b\arrow[d]\\
c\arrow[r,""{name=D,above}]\arrow[from=U,to=D,Rightarrow] & d,
\end{tikzcd}
\end{equation*}
which can be composed both horizontally and vertically and satisfy an evident \textit{interchange law}.

Every $\infty$-category equipped with a factorization system $\C^{\dagger}$ gives rise to a double $\infty$-category $\iota(\C^{\dagger})$ that has the same objects as $\C$, as horizontal resp. vertical morphisms the egressive resp. ingressive morphisms in $\C^{\dagger}$ and as squares the appropriate commutative squares in $\C$. Informally, this double $\infty$-category encodes how to obtain the egressive-ingressive factorization of a morphism composed \singlequote{the wrong way}.

\begin{introtheorem}[\protect{Un/Straightening for Factorization Systems, Theorem \ref{UnStraightening}}]\label{UnStraighteningIntro}
Let $\C^{\dagger}$ be an $\infty$-category equipped with a factorization system. Then there is a natural equivalence
\begin{equation*}
\Ambi(\C^{\dagger})^{\simeq}\simeq\Map_{\DCat_{\infty}}(\iota(\C^{\dagger})^{v-\op},\bbFact_{\infty})
\end{equation*}
between the space of ambifibrations over $\C^{\dagger}$ and the mapping space of double functors $\iota(\C^{\dagger})^{v-\op}\rightarrow\bbFact_{\infty}$, where $\bbFact_{\infty}$ is a certain double $\infty$-category of $\infty$-categories equipped with a factorization system.
\end{introtheorem}

The straightening of an ambifibration $p\colon\E^{\dagger}\rightarrow\C^{\dagger}$ is, roughly, a double functor $\iota(\C^{\dagger})^{v-\op}\rightarrow\bbFact_{\infty}$ mapping $c\mapsto\E_c^{\dagger}$, the fiber of $p$ over $c$, a horizontal resp. vertical morphism to the cocartesian resp. cartesian transport along the corresponding egressive resp. ingressive morphism of $\C^{\dagger}$ and which acts on squares as
\begin{equation*}
\begin{tikzcd}
y\arrow[r,"g",twoheadrightarrow] & z & {} & {}\arrow[d,phantom,""'{name=R}] & \E_y^{\dagger}\arrow[d,"f^{\ast}"]\arrow[r,"g_{\#}"] & \E_z^{\dagger}\arrow[d,"f^{\prime,\ast}"]\\
x\arrow[u,"f",rightarrowtail]\arrow[r,"g^{\prime}",twoheadrightarrow] & y^{\prime}\arrow[u,"f^{\prime}",rightarrowtail] & {}\arrow[u,phantom,""'{name=L}]\arrow[from=L,to=R,mapsto] & {} & \E_x^{\dagger}\arrow[r,"g^{\prime}_{\#}"]\arrow[ur,Rightarrow] & \E_{y^{\prime}}^{\dagger}.
\end{tikzcd}
\end{equation*}
The component of the natural transformation in the right-hand oplax square at an object $x\in\E_y$ is given by the map $g^{\prime}_{\#}f^{\ast}x\rightarrow f^{\prime,\ast}g_{\#}x$ in $\E_{y^{\prime}}$ obtained by the universal properties of (co-)cartesian morphisms from the composite $f^{\ast}x\rightarrowtail x\twoheadrightarrow g_{\#}x$ of a cartesian and a cocartesian morphism living over $gf\simeq f^{\prime}g^{\prime}$.

The proof of Theorem \ref{UnStraighteningIntro} proceeds by \textit{lifting} the double $\infty$-categorical un/straightening equivalence of Juran \cite[Theorem B]{Juran} to $\infty$-categories equipped with a factorization system. To do so, we combine Theorem \ref{FactorizationSystemIntro} with the following general theorem on lifting through faithful functors of $n$-uple $\infty$-categories, which may be of independent interest.

\begin{introtheorem}[Theorem \ref{FaithfulLifting}]\label{FaithfulLiftingIntro}
Let $f\colon\C\rightarrow\D$ be a faithful functor of $n$-uple $\infty$-categories. Then, for any functor $p\colon\T\rightarrow\D$ of $n$-uple $\infty$-categories, the joint evaluation map
\begin{equation*}
\Map_{n\Cat^{\uple}_{/\D}}(\T,\C)\hookrightarrow\prod_{[t]\in\pi_0(\T^{\simeq})}(\C_{p(t)})^{\simeq}
\end{equation*}
is an inclusion of discrete spaces. The image consists precisely of those families $(c_t)_{[t]\in\pi_0(\T^{\simeq})}$ satisfying:
\begin{itemize}
\item For every $\underline{\varepsilon}\in\{0,1\}^n\setminus\{(0,\dotsc,0)\}$ and $\underline{\varepsilon}$-cell $\alpha\colon\bbD^{\underline{\varepsilon}}\rightarrow\T$, its image $f\alpha$ together with the chosen lifts of objects is contained in the full subspace $\Map_{n\Cat^{\uple}}(\bbD^{\underline{\varepsilon}},\C)\hookrightarrow\Map_{n\Cat^{\uple}}(\bbD^{\underline{\varepsilon}},\D)\times_{\D^{\simeq,\times2^n}}\C^{\simeq,\times2^n}$.
\end{itemize}
Furthermore, if $1\le k\le n$ and $f\colon\C\rightarrow\D$ is $k$-faithful, it suffices to check the above condition for those $\underline{\varepsilon}\in\{0,1\}^n\setminus\{(0,\dotsc,0)\}$ such that $\abs{\underline{\varepsilon}}=k$. If $\C$ and $\D$ are $(\infty,n)$-categories, it suffices to check the above condition for $\underline{\varepsilon}$ of the form $(1,\dotsc,1,0,\dotsc,0)$.
\end{introtheorem}

The second question we study concerns the forgetful functor $p\colon\Fact_{\infty}\rightarrow\Cat_{\infty}$ from the $\infty$-category of $\infty$-categories equipped with a factorization system to that of $\infty$-categories. For (ordinary) categories, it is a result of Korostenski-Tholen \cite[Theorem B]{KorostenskiTholen} that this functor is a monadic right adjoint; the associated monad is $\Ar(-)$ with the multiplication given by composition. We extend this result to the present setting.

\begin{introtheorem}[\protect{Theorem \ref{MonadicityTheorem}}]\label{MonadicityIntro}
The forgetful functor $p\colon\Fact_{\infty}\rightarrow\Cat_{\infty}$ of $\infty$-categories is a monadic right adjoint.
\end{introtheorem}

\subsection*{Outline}

In Section \ref{FactSysDoubleCat}, we briefly recall some aspects of factorization systems \& double categories. In Section \ref{SegalSection}, we recall the general theory of Segal objects and prove the lifting criterion Theorem \ref{FaithfulLiftingIntro}. In Section \ref{AmbifibSection}, we carry out the technical construction of ambifibrations, proving Theorem \ref{FactorizationSystemIntro}. These inputs will be combined in Section \ref{UnStraighteningSection} to prove the un/straightening equivalence Theorem \ref{UnStraighteningIntro} as well as the variant Theorem \ref{UnStraighteningEg}. Independently of all the preceding sections, we prove Theorem \ref{MonadicityIntro} on monadicity in Section \ref{MonadicitySection}.

\subsection*{Conventions}

\begin{itemize}
\item The prefix \singlequote{$\infty$-} will be dropped henceforth, i.e. \singlequote{category} resp. \singlequote{$2$-category} resp. \singlequote{double category} always refers to $\infty$-category resp. $(\infty,2)$-category resp. double $\infty$-category.
\item The ($\infty$-)category of spaces/($\infty$-)groupoids will be referred to as the category of \textit{anima}, denoted $\An$.
\item The (large) ($\infty$-)category of (small) ($\infty$-)categories is denoted $\Cat$.
\item If $\C$ is an ($\infty$-)category and $f\colon x\rightarrow y$ a morphism in $\C$, then post- resp. pre-composition with $f$ will be denoted $f_{\circ}\colon\Map_{\C}(t,x)\rightarrow\Map_{\C}(t,y)$ resp. $f^{\circ}\colon\Map_{\C}(y,t)\rightarrow\Map_{\C}(x,t)$ for any $t\in\C$.
\item If $p\colon\E\rightarrow\C$ is a cocartesian resp. cartesian fibration of ($\infty$-categories) and $f\colon x\rightarrow y$ a morphism in $\C$, then we denote by $f_{\#}\colon\E_x\rightarrow\E_y$ resp. $f^{\ast}\colon\E_y\rightarrow\E_x$ the cocartesian resp. cartesian transport.
\item ($\infty$-)categories are either named or generically denoted by caligraphic letters $\C$, $\D$, etc.
\item ($\infty$-)categories that are equipped with a factorization system are denoted by a superscript $(-)^{\dagger}$.
\item ($\infty$-)$2$-categories resp. double ($\infty$-)categories are denoted in boldface resp. blackboard boldface.
\end{itemize}

\subsection*{Acknowledgments}
I would like to thank Georg Lehner as well as Marcus Nicolas for helpful conversations and João Lobo Fernandes for explaining his Example \ref{FernandesExample}, which inspired the construction in Section \ref{AmbifibSection}. Furthermore, I would like to thank my advisor Thomas Nikolaus for his support.

The author was funded by the Deutsche Forschungsgemeinschaft (DFG, German Research Foundation) – Project-ID 427320536 – SFB 1442, as well as under Germany’s Excellence Strategy EXC2044/2–390685587, Mathematics Münster: Dynamics–Geometry–Structure.

\section{Factorization Systems and Double Categories}\label{FactSysDoubleCat}

Let $\C$ be a category. Recall that a \textit{factorization system} on $\C$ consists of two equivalence-stable classes of morphisms in $\C$—abusing notation due to Barwick, we call these \textit{egressive} and \textit{ingressive} morphisms, denoted by $\twoheadrightarrow$ and $\rightarrowtail$, respectively—such that every morphism in $\C$ factors as an egressive morphism followed by an ingressive morphism and the two classes are \textit{orthogonal} in the sense that, for any egressive morphism $f\colon a\twoheadrightarrow b$ and ingressive morphism $g\colon x\rightarrowtail y$, the square of mapping anima
\begin{equation}\label{Orthogonality}
\begin{tikzcd}
\Map_{\C}(b,x)\arrow[r,"g_{\circ}"]\arrow[d,"f^{\circ}"]\pb & \Map_{\C}(b,y)\arrow[d,"f^{\circ}"]\\
\Map_{\C}(a,x)\arrow[r,"g_{\circ}"] & \Map_{\C}(a,y)
\end{tikzcd}
\end{equation}
is cartesian. The wide subcategories of $\C$ spanned by the egressive resp. ingressive morphisms are denoted $\C_{\eg}$ resp. $\C_{\ing}$. If $f\colon a\rightarrow b$ is any morphism such that the square $(\ref{Orthogonality})$ is cartesian for all ingressive morphisms $g\colon x\rightarrowtail y$, then $f$ is egressive and vice versa, i.e. $\C_{\eg}$ and $\C_{\ing}$ determine another \cite[Proposition 5.2.8.11]{HTT}.

Furthermore, recall that the datum of a factorization system is equivalent to a fully faifhful section of the composition map $d_1\colon\Fun([2],\C)\rightarrow\Fun([1],\C)$ whose essential image consists of the diagrams of the form $x\twoheadrightarrow y\rightarrowtail z$ \cite[Proposition 5.2.8.17]{HTT}. In other words, the section $\mathrm{fact}\colon\Fun([1],\C)\hookrightarrow\Fun([2],\C)$ provides functorial egressive-ingressive factorizations. For general references on factorization systems, we refer the reader to \cite[Section 3.1]{ABFJ}, \cite[Appendix B]{Soergel}, \cite[Section 10.3]{Haugseng}, \cite{BarkanSteinebrunner}.

\begin{definition}
Let $\Fact$ (resp. $\Fact^{\eg}$, $\Fact^{\ing}$) denote the full subcategory of $\Fun(\Lambda_2^2,\Cat)$ (resp. $\Fun([1],\Cat)$) on those diagrams $\C_{\eg}\hookrightarrow\C\hookleftarrow\C_{\ing}$ of two wide monomorphisms classifying the subcategories of egressive and ingressive morphisms in a factorization system on $\C$ (resp. those diagrams $\C_{\eg}\hookrightarrow\C$, $\C_{\ing}\hookrightarrow\C$ of a wide monomorphism classifying the subcategory of egressive, ingressive morphisms in a factorization system on $\C$). There are canonical forgetful functors $p\colon\Fact\rightarrow\Cat$ and $p^{\eg}\colon\Fact^{\eg}\rightarrow\Cat$ resp. $p^{\ing}\colon\Fact^{\ing}\rightarrow\Cat$.
\end{definition}
\begin{remark}
The involution $(-)^{\op}\colon\Cat\iso\Cat$ induces inverse equivalences $(-)^{\op}\colon\Fact^{\eg}\stackrel{\sim}{\longleftrightarrow}\Fact^{\ing}\colon(-)^{\op}$.
\end{remark}

\begin{proposition}\label{ForgetfulFunctors}
The forgetful functor $p\colon\Fact\rightarrow\Cat$ is conservative and faithful. For a category $\C$, the fiber $p^{-1}(\C)$ is discrete and equivalent to the set of factorization systems on $\C$. The forgetful functors $p^{\eg}\colon\Fact^{\eg}\rightarrow\Cat$ resp. $p^{\ing}\colon\Fact^{\ing}\rightarrow\Cat$ are faithful and, for a category $\C$, the fiber $(p^{\eg})^{-1}(\C)$ resp. $(p^{\ing})^{-1}(\C)$ is the poset of factorization systems on $\C$,  ordered by inclusion of egressive resp. ingressive morphisms, which are opposite.
\end{proposition}
\begin{proof}[Proof.]
The functors are faithful since factoring through a monomorphism is a condition, so the fibers are posets in any case. In particular, the descriptions of the fibers of $p^{\eg}$ resp. $p^{\ing}$ are automatic and these partial orders are opposite because $\C_{\eg}\subseteq\C_{\eg}^{\prime}$ is equivalent to $\C_{\ing}^{\prime}\subseteq\C_{\ing}$ by mutual orthogonality. Finally, these inequalities together imply equality, hence the fibers of $p$ are discrete and $p$ is conservative.
\end{proof}

\begin{proposition}\label{EgIsoFull}
The restriction functors $\Fact\rightarrow\Fact^{\eg}$ resp. $\Fact\rightarrow\Fact^{\ing}$ are wide monomorphisms. In particular, there are canonical equivalences $(\Fact^{\eg})^{\simeq}\simeq\Fact^{\simeq}\simeq(\Fact^{\ing})^{\simeq}$.
\end{proposition}
\begin{proof}[Proof.]
These functors are also faithful by left cancellation, so it remains to observe they are essentially surjective and iso-full. The former is immediate and the latter follows since an equivalence preserving egressive morphisms automatically preserves ingressive morphisms (resp. vice versa) by mutual orthogonality.
\end{proof}

\begin{proposition}[\protect{\cite[Proposition 5.2]{ChuHaugseng}}]\label{CatFactLimits}
The forgetful functor $p\colon\Fact\rightarrow\Cat$ creates small limits.
\end{proposition}
\begin{proof}[Proof.]
It suffices to show that the full subcategory $\Fact\subseteq\Fun(\Lambda_2^2,\Cat)$ is stable under limits. The limit in $\Fun(\Lambda_2^2,\Cat)$ of a diagram $(\C_{\bullet})_{\eg}\hookrightarrow\C_{\bullet}\hookleftarrow(\C_{\bullet})_{\ing}$ in $\Fact$ consists again of two wide monomorphisms $\C_{\eg}\hookrightarrow\C\hookleftarrow\C_{\ing}$, because monomorphisms are stable under limits and $(-)^{\simeq}$ preserves limits. These two wide subcategories are again orthogonal by taking the limit of the cartesian squares $(\ref{Orthogonality})$ as cartesian squares are stable under limits and mapping anima in limits are limits of mapping anima. Finally, because egressive-ingressive factorizations are unique up to contractible choice, they are coherently compatible with the diagram and induce egressive-ingressive factorizations in $\C$, hence $(\C_{\eg}\hookrightarrow\C\hookleftarrow\C_{\ing})\in\Fact$.
\end{proof}
\begin{remark}\label{RmkFilteredColimit}
The same analysis also applies to filtered colimits in $\Fact$.
\end{remark}

Additionally, we take note of the adjunctions
\begin{equation*}
\begin{tikzcd}[column sep=large]
\Cat\arrow[r,"(-)^{\flat}"{description},shift left=1em]\arrow[r,"(-)^{\sharp}"{description},shift right=1em] & \Fact^{\eg},\arrow[l,"p^{\eg}"{description}]\arrow[l,shift left=2em,"(-)_{\eg}"]\arrow[l,shift right=2em,"(-)\sqb{\eg^{-1}}"'] & \Cat\arrow[r,"(-)^{\sharp}"{description},shift left=1em]\arrow[r,"(-)^{\flat}"{description},shift right=1em] & \Fact^{\ing}\arrow[l,"p^{\ing}"{description}]\arrow[l,shift left=2em,"(-)_{\ing}"]\arrow[l,shift right=2em,"(-)\sqb{\ing^{-1}}"'].
\end{tikzcd}
\end{equation*}
These are obtained by restrcting similar adjunctions between $\Cat$ and the category of marked categories. Explicitly, for a category $\C$, the factorization systems $\C^{\flat}$ and $\C^{\sharp}$ are those in which all morphisms are ingressive resp. egressive (the other class then necessarily consists of precisely the equivalences, by orthogonality) and we call these the \textit{trivial} factorization systems. Further restricting these, we also obtain adjunctions
\begin{equation*}
\begin{tikzcd}[column sep=large]
\Cat\arrow[r,"(-)^{\flat}"{description}] & \Fact,\arrow[l,shift right=1em,"(-)\sqb{\eg^{-1}}"']\arrow[l,shift left=1em,"(-)_{\ing}"] & \Cat\arrow[r,"(-)^{\sharp}"{description}] & \Fact.\arrow[l,shift right=1em,"(-)\sqb{\ing^{-1}}"']\arrow[l,shift left=1em,"(-)_{\eg}"]
\end{tikzcd}
\end{equation*}


\begin{example}
The category $\Fact$ is presentable, but the category $\Fact^{\eg}$ is neither complete nor cocomplete. To see this, let $f\colon\C\rightarrow\D$ be a functor of such that the preimage $f^{-1}(\D^{\simeq})\hookrightarrow\C$ is not the class of egressive morphisms in any factorization system on $\C$, e.g. take $[1]^2\rightarrow[2]$ classifying
\begin{equation*}
\begin{tikzcd}
0\arrow[r,equals]\arrow[d] & 0\arrow[d]\\
1\arrow[r] & 2.
\end{tikzcd}
\end{equation*}
The above adjunctions imply that limits in $\Fact^{\eg}$ are computed underlying on both the category itself and the wide subcategory of egressive morphisms, so the pullback of the cospan $\C^{\sharp}\rightarrow\D^{\sharp}\leftarrow\D^{\flat}$ would have to be $(\C,f^{-1}(\D^{\simeq}))$, yet this is not an object in $\Fact^{\eg}$, so the pullback does not exist.

The adjunctions also imply that colimits in $\Fact^{\eg}$ are computed underlying, so the pushout of the span $f^{-1}(\D^{\simeq})^{\sharp}\leftarrow f^{-1}(\D^{\simeq})^{\flat}\rightarrow\C^{\flat}$ would have to be of the form $(\C,\C_{\eg})$. There are functors $f^{-1}(\D^{\simeq})^{\sharp}\rightarrow(\C,\C_{\eg})$ by definition and $(\C,\C_{\eg})\rightarrow\C[f^{-1}(\D^{\simeq})^{-1}]^{\flat}$ by the universal property, forcing $f^{-1}(\D^{\simeq})=\C_{\eg}$. Thus the pushout does not exist either.
\end{example}

Next, we summarize the relevant background on $2$- and double categories that we will need later on, as well as Juran's results from \cite{Juran}. The reader is referred to the references for more complete introductions.

\begin{definition}
A \textit{double category} is a $2$-uple complete Segal anima (cf. Section \ref{SegalSection}). A \textit{$2$-category} is a double category $X_{\bullet,\bullet}\colon\Delta^{\times2,\op}\rightarrow\An$ such that the simplicial object $X_{0,\bullet}\colon\Delta^{\op}\rightarrow\An$ is constant. The full subcategories of $\Fun(\Delta^{\times2,\op},\An)$ on the $2$- resp. double categories are denoted $2\Cat\subseteq\DCat$.
\end{definition}

\begin{notation}
The involutions induced by $\op\times\id_{\Delta},\id_{\Delta}\times\op\colon\Delta\times\Delta\rightarrow\Delta\times\Delta$ on $\Fun(\Delta^{\times2,\op},\An)$ restrict to $\DCat$ and are denoted $(-)^{h-\op},(-)^{v-\op}\colon\DCat\rightarrow\DCat$ respectively. The involution induced by the swap $\tau\colon\Delta\times\Delta\rightarrow\Delta\times\Delta,\,[n,m]\mapsto[m,n]$ similarly restricts to an involution $(-)^t\colon\DCat\rightarrow\DCat$.

Furthermore, the inclusions $(-)_h\colon2\Cat\hookrightarrow\DCat$ resp. $(-)_v\coloneqq(-)^t\circ(-)_h\colon2\Cat\hookrightarrow\DCat$ admit right adjoints $\Hor(-)\colon\DCat\rightarrow2\Cat$ resp. $\Ver(-)\colon\DCat\rightarrow2\Cat$ called the \textit{horizontal} resp. \textit{vertical fragment}.
\end{notation}

\begin{definition}\label{LocallyFullyFaithfulDefn}
A functor of $2$-categories is called \textit{fully faithful} resp. \textit{locally fully faithful} resp. \textit{faithful} if it induces on Hom-categories an equivalence resp. fully faithful functor resp. monomorphism. It is \textit{wide} resp. \textit{locally wide} if it is essentially surjective resp. essentially surjective on $\Hom$-categories.
\end{definition}

\begin{recollection}
The category $2\Cat$ of $2$-categories has a canonical monoidal structure, the (oplax) \textit{Gray tensor product} $\boxtimes\colon2\Cat\times2\Cat\rightarrow2\Cat,\,(\bm{\C},\bm{\D})\mapsto\bm{\C}\boxtimes\bm{\D}$; see \cite{LoubatonRuit} for a review and further references. This monoidal structure has the following properties:
\begin{itemize}
\item It is compatible with colimits in either variable.
\item It restricts on the subcategory $\square$ of \textit{Gray cubes} \cite[Definition 1.5]{CampionMaehara} to the strict Gray tensor product, which admits a combinatorially explicit description \cite[Definition 1.4]{CampionMaehara}.
\item It is biclosed. More precisely, there are adjunctions
\begin{equation*}
		(-)\boxtimes\bm{\D}\dashv\bfFun^{\oplax}(\bm{\D},-),\qquad\bm{\C}\boxtimes(-)\dashv\bfFun^{\lax}(\bm{\C},-),
\end{equation*}
where $\bfFun^{(\op)\lax}$ is the $2$-category of functors, (op)lax natural transformations and modifications.
\end{itemize}
\end{recollection}
\begin{remark}
The first two conditions uniquely determine the Gray tensor product \cite[Corollary 3.5, (2)]{CampionMaehara}.
\end{remark}
\begin{remark}
In particular, the first two conditions imply that the unit of this monoidal structure is the terminal $2$-category and so the unique maps $\bm{\C}\rightarrow\ast$ and $\bm{\D}\rightarrow\ast$ induce a natural map $\bm{\C}\boxtimes\bm{\D}\rightarrow\bm{\C}\times\bm{\D}$. This map is an equivalence if $\bm{\C}$ or $\bm{\D}$ is terminal, so it is an equivalence if $\bm{\C}$ or $\bm{\D}$ is an anima by compatibility with colimits.
\end{remark}

Later on, we will need a rudimentary description of the objects and morphisms in the Gray tensor product.

\begin{proposition}\label{GrayRudimentary}
Let $\bm{\C},\bm{\D}$ be $2$-categories. Then:
\begin{enumerate}[label=(\arabic*)]
\item The natural map $\pi_0((\bm{\C}\boxtimes\bm{\D})^{\simeq})\rightarrow\pi_0((\bm{\C}\times\bm{\D})^{\simeq})\cong\pi_0(\bm{\C}^{\simeq})\times\pi_0(\bm{\D}^{\simeq})$ is a bijection. A pair $(c,d)\in\pi_0(\bm{\C}^{\simeq})\times\pi_0(\bm{\D}^{\simeq})$ corresponds to the object $c\boxtimes d$ classified by the functor $\ast\simeq\ast\boxtimes\ast\stackrel{c\boxtimes d}{\rightarrow}\bm{\C}\boxtimes\bm{\D}$.
\item The $1$-morphisms of $\bm{\C}\boxtimes\bm{\D}$ are generated under composition by those of the form $f\boxtimes d\colon c\boxtimes d\rightarrow c^{\prime}\boxtimes d$ for $f\colon c\rightarrow c^{\prime}$ a morphism in $\bm{\C}$ and $d\in\bm{\D}$, classified by the functor $[1]\simeq[1]\boxtimes\ast\stackrel{f\boxtimes d}{\rightarrow}\bm{\C}\boxtimes\bm{\D}$, and, dually, those of the form $c\boxtimes g\colon c\boxtimes d\rightarrow c\boxtimes d^{\prime}$ for $c\in\bm{\C}$ and $g\colon d\rightarrow d^{\prime}$ a morphism in $\bm{\D}$.
\end{enumerate}
\end{proposition}
\begin{proof}[Proof.]
The full subcategory inclusion $\square\subseteq2\Cat$ is dense \cite[Corollary 2.11]{CampionMaehara}, hence there is a natural equivalence $\bm{\E}\simeq\colim_{\square^d\rightarrow\bm{\E}}\square^d$ for any $2$-category $\bm{\E}$ (more precisely, the colimit is indexed by the relative slice $\square\times_{2\Cat}2\Cat_{/\bm{\E}}$). In particular, we obtain
\begin{equation*}
\bm{\C}\boxtimes\bm{\D}\simeq\colim_{\square^d\rightarrow\bm{\C},\square^{d^{\prime}}\rightarrow\bm{\D}}\square^d\boxtimes\square^{d^{\prime}}.
\end{equation*}
Then, using a general description of colimits in $2\Cat$ (e.g. using Segalification \cite[Section 2.2]{Segalification} and completion \cite[Sections 1.2-1.3]{Lurie}), we see that $(2)$ reduces to the case $\bm{\C}=\square^d$ and $\bm{\D}=\square^{d^{\prime}}$ for $d,d^{\prime}\ge0$, which is true by the combinatorially explicit description of \cite[Definitions 1.4-1.5]{CampionMaehara}. The same reduction also demonstrates that every object of $\bm{\C}\boxtimes\bm{\D}$ is of the form $c\boxtimes d$ for some $c\in\bm{\C}$ and $d\in\bm{\D}$.

To prove $(1)$, observe that the maps $\bm{\C}^{\simeq}\hookrightarrow\bm{\C}$ and $\bm{\D}^{\simeq}\hookrightarrow\bm{\D}$ induce a map $\bm{\C}^{\simeq}\times\bm{\D}^{\simeq}\simeq\bm{\C}^{\simeq}\boxtimes\bm{\D}^{\simeq}\rightarrow\bm{\C}\boxtimes\bm{\D}$, which factors as $\bm{\C}^{\simeq}\times\bm{\D}^{\simeq}\rightarrow(\bm{\C}\boxtimes\bm{\D})^{\simeq}$ and is a section of the natural map $(\bm{\C}\boxtimes\bm{\D})^{\simeq}\rightarrow\bm{\C}^{\simeq}\times\bm{\D}^{\simeq}$. The first map is essentially surjective by the preceding discussion, hence these induce inverse bijections on $\pi_0$.
\end{proof}

Furthermore, we recall the following theorem, although we will only need the construction of the functor.

\begin{theorem}
The bicosimplicial object $\Delta\times\Delta\rightarrow2\Cat,\,[n,m]\mapsto[m]\boxtimes[n]$ induces a fully faithful functor
\begin{equation*}
\Sq\colon2\Cat\hookrightarrow\DCat,\,\C\mapsto([n,m]\mapsto\Map_{2\Cat}([m]\boxtimes[n],\C)).
\end{equation*}
This is called the (oplax) \textnormal{squares embedding}.\qed
\end{theorem}

\begin{notation}
Let $\bfCat$ denote the $2$-category of categories, functors and natural transformations and $\bbCat$ the double category $\Sq(\bfCat)$ whose objects are categories, horizontal resp. vertical morphisms are functors and squares are oplax commutative squares in $\bfCat$, i.e.
\begin{equation*}
\begin{tikzcd}
\C\arrow[r]\arrow[d] & \C^{\prime}\arrow[d]\\
\D\arrow[r]\arrow[ur,Rightarrow] & \D^{\prime}.
\end{tikzcd}
\end{equation*}
\end{notation}

Finally, we are able to recall Juran's double-categorical framework and un/straightening equivalence.

\begin{theorem}[\protect{\cite[Theorem 3.19]{Juran}}]
The bicosimplicial object $\Delta^{\times2}\rightarrow\Fact,\,[n,m]\mapsto[n]^{\sharp}\times[m]^{\flat}$ (cf. \cite[Notation 2.14]{Juran}) induces a fully faithful functor
\begin{equation*}
\iota\colon\Fact\hookrightarrow\DCat,\,\C^{\dagger}\mapsto([n,m]\mapsto\Map_{\Fact}([n]^{\sharp}\times[m]^{\flat},\C^{\dagger})).\tag*{\qedsymbol}
\end{equation*}
\end{theorem}

\begin{theorem}[\protect{\cite[Proposition 4.5, Theorem 4.6]{Juran}}]\label{JuranUnStraightening}
Let $\C^{\dagger}$ be a category equipped with a factorization system. Then there is a natural equivalence
\begin{equation*}
\Ortho(\C^{\dagger})^{\simeq}\simeq\Map_{\DCat}(\iota(\C^{\dagger})^{v-\op},\bbCat)
\end{equation*}
between the anima of \textnormal{curved orthofibrations} over $\C^{\dagger}$ (see Definition \ref{AmbiFibDefn}) and the mapping anima of double functors $\iota(\C^{\dagger})^{v-\op}\rightarrow\bbCat$.\qed
\end{theorem}

\section{Segal Objects and Faithfulness}\label{SegalSection}

In this section, we recall the formalism of Segal objects to study faithful maps and prove Theorem \ref{FaithfulLiftingIntro} on lifting through faithful functors of $n$-uple categories. This is a categorification of covering space theory.

\begin{definition}
Let $\C$ be a finitely complete category. A \textit{Segal object} in $\C$ is a simplicial object $X\colon\Delta^{\op}\rightarrow\C$ that satisfies the \textit{Segal condition}, i.e. the canonical map $X_n\rightarrow X_1\times_{X_0}\dotsc\times_{X_0}X_1$ is an equivalence for all $n\ge2$. It is called a \textit{groupoid object} if the canonical map $X_n\rightarrow X_S\times_{X_{\{s\}}}X_{S^{\prime}}$ is an equivalence for every partition $[n]=S\cup S^{\prime}$ with $S\cap S^{\prime}=\{s\}$. The full subcategories of $\Fun(\Delta^{\op},\C)$ on the groupoid resp. Segal objects are denoted $\Gpd(\C)\subseteq\Seg(\C)$.
\end{definition}

\begin{proposition}[\protect{\cite[Proposition 1.1.14]{Lurie}}]\label{GroupoidCore}
Let $\C$ be a finitely complete category. The inclusion $\Gpd(\C)\subseteq\Seg(\C)$ admits a right adjoint $(-)^{\simeq}\colon\Seg(\C)\rightarrow\Gpd(\C)$. The counit map $X^{\simeq}\hookrightarrow X$ is a monomorphism and induces an equivalence $X^{\simeq}_0\iso X_0$ for all $X\in\Seg(\C)$.\qed
\end{proposition}

\begin{definition}
A Segal object $X$ in a finitely complete category $\C$ is \textit{complete} if the groupoid object $X^{\simeq}$ is constant. The full subcategory of $\Seg(\C)$ on the complete Segal objects is denoted $\Cat(\C)$. The forgetful functor $\ev_0\colon\Cat(\C)\rightarrow\C$ is then also denoted $(-)^{\simeq}\colon\Cat(\C)\rightarrow\C$.
\end{definition}
\begin{remark}
For $\C=\An$, the Rezk nerve $N\colon\Cat\rightarrow\Fun(\Delta^{\op},\An),\,\C\mapsto([n]\mapsto\Map_{\Cat}([n],\C))$ is fully faithful and induces an equivalence onto the full subcategory of complete Segal anima (\cite{JoyalTierney}, \cite{HebestreitSteinebrunner}).
\end{remark}

\begin{remark}\label{CompletenessCharacterization}
Let $X$ be a Segal object in $\C$ and consider the pullback
\begin{equation*}
\begin{tikzcd}
X_{\eq}\arrow[r]\arrow[d] & X_3\arrow[d,"{(d_{\{1,3\}},d_{\{0,2\}})}"]\\
X_0\times X_0\arrow[r,"s_0\times s_0"] & X_1\times X_1.
\end{tikzcd}
\end{equation*}
The canonical map $X_{\eq}\rightarrow X_3\stackrel{d_{\{0,3\}}}{\rightarrow}X_1$ is a monomorphism by \cite[Proposition 1.1.14]{Lurie}, the degeneracy factors through this monomorphism as a map $s_0\colon X_0\rightarrow X_{\eq}$ and loc. cit. implies that $X$ is complete if and only if this map is an equivalence.
\end{remark}

\begin{definition}
For $n\in\N$, we inductively define an $n$-uple (complete) Segal object in a finitely completely category $\C$ to be a (complete) Segal object in $(n-1)$-uple (complete) Segal objects in $\C$ if $n\ge1$. The full subcategory of $\Fun(\Delta^{\times n,\op},\C)$ on the $n$-uple (complete) Segal objects is denoted $\Seg^n(\C)$ resp. $\Cat^n(\C)$.
\end{definition}
\begin{remark}
For $\C=\An$, we denote $n\Cat^{\uple}\coloneqq\Cat^n(\An)$ and consider it as the category of \textit{$n$-uple categories}. The category $n\Cat$ of \textit{$n$-categories} is the full subcategory of $n\Cat^{\uple}$ on those $n$-uple categories that satisfy the \textit{globularity} condition: for every $k_1,\dotsc,k_{i-1}\ge0$, the $(n-i)$-uple simplicial anima $\C_{k_1,\dotsc,k_{i-1},0,\bullet,\dotsc,\bullet}$ is constant (see \cite[Section 1.3]{Lurie}, \cite[Section 14]{BarwickSchommerPries} and cf. \cite[Lemma 2.8]{JohnsonFreydScheimbauer}).
\end{remark}

\begin{proposition}\label{SegAdjunction}
Let $\C$ be a finitely complete category and $n\ge0$. There is an adjunction
\begin{equation*}
\begin{tikzcd}[column sep=5em]
\Seg^n(\C)\arrow[r,"\ev_{0,\dotsc,0}"{description}]\arrow[r,shift left=0.75em,hookleftarrow,"\const"]\arrow[r,shift right=0.75em,hookleftarrow,"T"'] & \C,
\end{tikzcd}
\end{equation*}
in which $\ev_{0,\dotsc,0}$ is both a reflective and a coreflective localization.
\end{proposition}
\begin{proof}[Proof.]
The evaluation functor $\ev_{0,\dotsc,0}\colon\Fun(\Delta^{\times n,\op},\C)\rightarrow\C$ is equivalently the limit functor and admits a left adjoint $\const$ taking constant diagrams, which clearly factors through the full subcategory on the $n$-uple Segal objects. To see that $\ev_{0,\dotsc,0}$ has a right adjoint, consider the pointwise formula for right Kan extensions: the right Kan extension of $X\colon\{[0,\dotsc,0]\}\rightarrow\C$ along $\{[0,\dotsc,0]\}\hookrightarrow\Delta^{\times n,\op}$ is given by
\begin{equation*}
[k_1,\dotsc,k_n]\mapsto\lim_{(\Delta^{\times n,\op})_{[k_1,\dotsc,k_n]/}\times_{\Delta^{\times n,\op}}\{[0,\dotsc,0]\}}X\simeq\lim_{\prod_{i=1}^n\Map_{\Delta^{\op}}([0],[k_i])}X\simeq X^{\times(k_1+1)\dotsc(k_n+1)},
\end{equation*}
which exists by finite completeness of $\C$ and factors through the full subcategory on the $n$-uple Segal objects by inspection. Finally, note that $\const$ is fully faithful because $\Delta^{\times n,\op}$ is weakly contractible and $T$ is then fully faithful by abstract nonsense.
\end{proof}

\begin{definition}\label{FaithfulSegal}
A map $f\colon X\rightarrow Y$ of $n$-uple Segal objects is called \textit{faithful} resp. \textit{fully faithful} if the induced map $f^{\prime}\colon X\rightarrow Y\times_{TY_{0,\dotsc,0}}TX_{0,\dotsc,0}$ is a monomorphism resp. equivalence in $\Seg^n(\C)$.
\end{definition}

\begin{lemma}
The class of (fully) faithful functors contains the equivalences, is stable under under limits in the arrow category, in particular under base-change, and satisfies left cancellation.\qed
\end{lemma}

\begin{proposition}\label{FaithfulVsMono}
Let $f\colon X\rightarrow Y$ be a map of $n$-uple Segal objects in a finitely completely category $\C$. Then $f$ is a monomorphism if and only if it is faithful and $f_{0,\dotsc,0}\colon X_{0,\dotsc,0}\rightarrow Y_{0,\dotsc,0}$ is a monomorphism.
\end{proposition}
\begin{proof}[Proof.]
The map $f$ is a composite of $f^{\prime}\colon X\rightarrow Y\times_{TY_{0,\dotsc,0}}TX_{0,\dotsc,0}$ and the map $Y\times_{TY_{0,\dotsc,0}}TX_{0,\dotsc,0}\rightarrow Y$ base-changed from $Tf_{0,\dotsc,0}\colon TX_{0,\dotsc,0}\rightarrow TY_{0,\dotsc,0}$. In particular, the latter map is a monomorphism if $f$ is a monomorphism since $T$ and $\ev_{0,\dotsc,0}$ are right adjoints, hence preserve monomorphisms. In this case, we see that $f$ is a monomorphism if and only if $f^{\prime}$ is a monomorphism, i.e. $f$ is faithful, by left cancellation.
\end{proof}

\begin{proposition}\label{FaithfulInductive}
Let $f\colon X\rightarrow Y$ be a map of $n$-uple Segal objects in a finitely complete category $\C$. Then $f$ is (fully) faithful if and only if it is (fully) faithful as a map of Segal objects in $\Seg^{n-1}(\C)$ and the underlying map $i_0^{\ast}f$ of $(n-1)$-uple Segal objects in $\C$ is (fully) faithful.
\end{proposition}
\begin{proof}[Proof.]
Consider the adjunctions provided by Proposition \ref{SegAdjunction}:
\begin{equation*}
\begin{tikzcd}[column sep=5em]
\Seg^n(\C)\arrow[r,"i_0^{\ast}"{description}]\arrow[r,shift left=0.75em,hookleftarrow,"\const"]\arrow[r,shift right=0.75em,hookleftarrow,"T^{\prime}"'] & \Seg^{n-1}(\C)\arrow[r,"\ev_{0,\dotsc,0}"{description}]\arrow[r,shift left=0.75em,hookleftarrow,"\const"]\arrow[r,shift right=0.75em,hookleftarrow,"T_{n-1}"'] & \C,
\end{tikzcd}
\end{equation*}
where we denote the composite right adjoint $T_n\simeq T^{\prime}T_{n-1}$. Then we can factor the induced map $f^{\prime}$ as
\begin{equation*}
\begin{tikzcd}
X\arrow[r] & Y\times_{T^{\prime}i_0^{\ast}Y}T^{\prime}i_0^{\ast}X\arrow[r]\arrow[d]\pb & Y\times_{T_nY_{0,\dotsc,0}}T_nX_{0,\dotsc,0}\arrow[d]\\
& T^{\prime}i_0^{\ast}X\arrow[r,"T^{\prime}(i_0^{\ast}f)^{\prime}"] & T^{\prime}i_0^{\ast}Y\times_{T^{\prime}T_{n-1}Y_{0,\dotsc,0}}T^{\prime}T_{n-1}Y_{0,\dotsc,0}.
\end{tikzcd}
\end{equation*}
The map $(i_0^{\ast}f)^{\prime}\colon i_0^{\ast}X\rightarrow i_0^{\ast}Y\times_{T_{n-1}Y_{0,\dotsc,0}}T_{n-1}X_{0,\dotsc,0}$ is obtained by applying $i_0^{\ast}$ to the map $f^{\prime}\colon X\rightarrow Y\times_{T_nY_{0,\dotsc,0}}T_nX_{0,\dotsc,0}$. In particular, if $f$ is (fully) faithful, then $i_0^{\ast}f$ is (fully) faithful. In this case, the right horizontal maps in the above diagram are monomorphisms resp. equivalences, so the left map is a monomorphism resp. equivalence if and only if $f^{\prime}$ is a monomorphism resp. equivalence by left cancellation.
\end{proof}

\begin{proposition}\label{FaithfulLiftingSegal}
Let $\C$ be a finitely complete category and $f\colon X\rightarrow Y$ a faithful map of $n$-uple Segal objects in $\C$. Then, for any map $Z\rightarrow Y$ of $n$-uple Segal objects in $\C$, the induced map
\begin{equation*}
\begin{tikzcd}
\ev_{0,\dotsc,0}\colon\Map_{\Seg^n(\C)_{/Y}}(Z,X)\arrow[r,hookrightarrow] & \Map_{\C_{/Y_{0,\dotsc,0}}}(Z_{0,\dotsc,0},X_{0,\dotsc,0})
\end{tikzcd}
\end{equation*}
is a monomorphism of anima. The image consists of precisely those components on the maps $Z_{0,\dotsc,0}\rightarrow X_{0,\dotsc,0}$ over $Y_{0,\dotsc,0}$ such that the induced map $Z\rightarrow Y\times_{TY_{0,\dotsc,0}}TX_{0,\dotsc,0}$ factors through $X\hookrightarrow Y\times_{TY_{0,\dotsc,0}}TX_{0,\dotsc,0}$.
\end{proposition}
\begin{proof}[Proof.]
The map $f^{\prime}\colon X\hookrightarrow Y\times_{TY_0}TX_0$ is a monomorphism in $\Seg^n(\C)$, hence also in $\Seg^n(\C)_{/Y}$, and we obtain
\begin{equation*}
\begin{array}{rcl}
\Map_{\Seg^n(\C)_{/Y}}(Z,X) & \hookrightarrow & \Map_{\Seg^n(\C)_{/Y}}(Z,Y\times_{TY_{0,\dotsc,0}}TX_{0,\dotsc,0})\\
 & \simeq & \Map_{\Seg^n(\C)_{/TY_{0,\dotsc,0}}}(Z,TX_{0,\dotsc,0})\\
 & \simeq & \Map_{\C_{/Y_{0,\dotsc,0}}}(Z_{0,\dotsc,0},X_{0,\dotsc,0}).
\end{array}
\end{equation*}
The characterization of the essential image is tautological.
\end{proof}

Now, we specialize to $n$-uple \textit{complete} Segal objects.

\begin{proposition}\label{FaithfulIsTruncatedOnCore}
Let $f\colon X\rightarrow Y$ be a faithful resp. fully faithful map of complete Segal objects in a finitely complete category $\C$. Then $f_0\colon X_0\rightarrow Y_0$ is $0$-truncated resp. $(-1)$-truncated (i.e. a monomorphism).
\end{proposition}
\begin{proof}[Proof.]
The functor $T\colon\C\rightarrow\Seg(\C)$ factors through $\Gpd(\C)$, so we obtain $(f^{\simeq})^{\prime}\colon X^{\simeq}\rightarrow Y^{\simeq}\times_{TY_0}TX_0$ by applying $(-)^{\simeq}$ to $f^{\prime}\colon X\rightarrow Y\times_{TY_0}TX_0$, i.e. $f^{\simeq}\colon X^{\simeq}\rightarrow Y^{\simeq}$ is (fully) faithful if $f\colon X\rightarrow Y$ is, but this map is the diagonal $X_0\rightarrow Y_0\times_{Y_0^{\times2}}X_0^{\times2}\simeq X_0\times_{Y_0}X_0$ by completeness, hence $f_0$ is $0$- resp. $(-1)$-truncated.
\end{proof}

\begin{corollary}\label{FaithfulOnObjects}
Let $f\colon X\rightarrow Y$ be a faithful resp. fully faithful map of $n$-uple complete Segal objects in a finitely complete category $\C$. Then $f_{0,\dotsc,0}\colon X_{0,\dotsc0}\rightarrow Y_{0,\dotsc,0}$ is $0$-truncated resp. $(-1)$-truncated (i.e. a monomorphism).
\end{corollary}
\begin{proof}[Proof.]
The induced map $f_{0,\dotsc,0,\bullet}\colon X_{0,\dotsc,0,\bullet}\rightarrow Y_{0,\dotsc,0,\bullet}$ is a faithful resp. fully faithful map of complete Segal objects in $\C$ by the argument of Proposition \ref{FaithfulInductive}. Then Proposition \ref{FaithfulIsTruncatedOnCore} applies.
\end{proof}
\begin{corollary}\label{FaithfulIsTruncated}
Faithful resp. fully faithful maps of $n$-uple complete Segal objects in a finitely complete category $\C$ are $0$-truncated resp. $(-1)$-truncated (i.e. monomorphisms).
\end{corollary}
\begin{proof}[Proof.]
Fully faithful maps of $n$-uple complete Segal objects are monomorphisms by combining Proposition \ref{FaithfulVsMono} and Corollary \ref{FaithfulOnObjects}. It is straightforward to check that the diagonal of a faithful map is a fully faithful map using that $T$ is a right adjoint, hence the former is $0$-truncated.
\end{proof}
\begin{corollary}
Let $\C$ be a finitely complete category and $f\colon X\rightarrow Y$ a faithful map of $n$-uple complete Segal objects in $\C$. Then, for any map $Z\rightarrow Y$ of $n$-uple Segal objects in $\C$, $\Map_{\Seg^n(\C)_{/Y}}(Z,X)$ is discrete.\qed
\end{corollary}

Finally, we specialize to the case $\C=\An$. For any tuple $\underline{\varepsilon}=(\varepsilon_1,\dotsc,\varepsilon_n)\in\{0,1\}^n$, we consider it as an object of $\Delta^{\times n}$ and denote the associated representable presheaf by $\bbD^{\underline{\varepsilon}}$. This is an $n$-uple category called the \singlequote{walking $\underline{\varepsilon}$-cell}. Its \textit{dimension} is the number $\abs{\underline{\varepsilon}}\coloneqq\sum_{i=1}^n\varepsilon_n$. The walking $\underline{\varepsilon}$-cell has a skeletal filtration
\begin{equation*}
\begin{tikzcd}
\bigsqcup\limits_{2^{\abs{\underline{\varepsilon}}}}\ast\simeq\partial_{\le0}\bbD^{\underline{\varepsilon}}\arrow[r,hookrightarrow] & \partial_{\le1}\bbD^{\underline{\varepsilon}}\arrow[r,hookrightarrow] & \dotsc\arrow[r,hookrightarrow] & \partial_{\le\abs{\underline{\varepsilon}}-1}\bbD^{\underline{\varepsilon}}\arrow[r,hookrightarrow] & \bbD^{\underline{\varepsilon}}.
\end{tikzcd}
\end{equation*}
The \textit{$k$-skeleton} $\partial_{\le k}\bbD^{\underline{\varepsilon}}$ is the union of the $\le k$-dimensional representable sub-presheaves of $\bbD^{\underline{\varepsilon}}$ and again an $n$-uple category. The codimension $1$-skeleton $\partial\bbD^{\underline{\varepsilon}}\coloneqq\partial_{\le\abs{\underline{\varepsilon}}-1}\bbD^{\underline{\varepsilon}}$ is also called the \textit{boundary} of $\bbD^{\underline{\varepsilon}}$. The crucial property of the skeletal filtration is that the $k$-th step $\partial_{\le k-1}\bbD^{\underline{\varepsilon}}\hookrightarrow\partial_{\le k}\bbD^{\underline{\varepsilon}}$ is the cobase-change of a (finite) coproduct of boundary inclusions $\partial\bbD^{\underline{\varepsilon}^{\prime}}\hookrightarrow\bbD^{\underline{\varepsilon}^{\prime}}$ with $\abs{{\underline{\varepsilon}}^{\prime}}=k$, which is already true in $\PSh(\Delta^{\times n})$.

\begin{proposition}\label{FaithfulCells}
A functor $f\colon\C\rightarrow\D$ of $n$-uple Segal anima is faithful resp. fully faithful if and only if for every $\underline{\varepsilon}\in\{0,1\}^n\setminus\{(0,\dotsc,0)\}$ and $\alpha\colon\partial\bbD^{\underline{\varepsilon}}\rightarrow\C$, the canonical map
\begin{equation*}
\begin{tikzcd}
\Map_{n\Cat^{\uple}}(\bbD^{\underline{\varepsilon}},\C)\times_{\Map_{n\Cat^{\uple}}(\partial\bbD^{\underline{\varepsilon}},\C)}\{\alpha\}\arrow[r] & \Map_{n\Cat^{\uple}}(\bbD^{\underline{\varepsilon}},\D)\times_{\Map_{n\Cat^{\uple}}(\partial\bbD^{\underline{\varepsilon}},\D)}\{f\alpha\}
\end{tikzcd}
\end{equation*}
is a monomorphism resp. equivalence.
\end{proposition}
\begin{proof}[Proof.]
That $f^{\prime}\colon\C\rightarrow\D\times_{T\D_{0,\dotsc,0}}T\C_{0,\dotsc,0}$ is a monomorphism resp. equivalence can be checked pointwise and the Segal condition reduces this to the points $\underline{\varepsilon}\in\{0,1\}^n$, i.e. to applying $\Map_{n\Cat^{\uple}}(\bbD^{\underline{\varepsilon}},-)$. This is always true for $\underline{\varepsilon}=(0,\dotsc,0)$ and we may use the skeletal filtration to factor $\Map_{n\Cat^{\uple}}(\bbD^{\underline{\varepsilon}},f^{\prime})$ as a composite
\begin{equation*}
\begin{array}{rcl}
\Map_{n\Cat^{\uple}}(\bbD^{\underline{\varepsilon}},\C) & \rightarrow & \Map_{n\Cat^{\uple}}(\bbD^{\underline{\varepsilon}},\D)\times_{\Map_{n\Cat^{\uple}}(\partial_{\le\abs{\underline{\varepsilon}}-1}\bbD^{\underline{\varepsilon}},\D)}\Map_{n\Cat^{\uple}}(\partial_{\le\abs{\underline{\varepsilon}}-1}\bbD^{\underline{\varepsilon}},\C)\\
 & \rightarrow & \Map_{n\Cat^{\uple}}(\bbD^{\underline{\varepsilon}},\D)\times_{\Map_{n\Cat^{\uple}}(\partial_{\le\abs{\underline{\varepsilon}}-2}\bbD^{\underline{\varepsilon}},\D)}\Map_{n\Cat^{\uple}}(\partial_{\le\abs{\underline{\varepsilon}}-2}\bbD^{\underline{\varepsilon}},\C)\\
& & \vdots\\
 & \rightarrow & \Map_{n\Cat^{\uple}}(\bbD^{\underline{\varepsilon}},\D)\times_{\Map_{n\Cat^{\uple}}(\partial_{\le0}\bbD^{\underline{\varepsilon}},\D)}\Map_{n\Cat^{\uple}}(\partial_{\le0}\bbD^{\underline{\varepsilon}},\C)\\
& \simeq & \Map_{n\Cat^{\uple}}(\bbD^{\underline{\varepsilon}},\D)\times_{\Map_{n\Cat^{\uple}}(\bbD^{\underline{\varepsilon}},T\D_{0,\dotsc,0})}\Map_{n\Cat^{\uple}}(\bbD^{\underline{\varepsilon}},T\C_{0,\dotsc,0}).
\end{array}
\end{equation*}
The individual maps are base-changed from (finite) products of the analogous maps
\begin{equation*}
\Map_{n\Cat^{\uple}}(\bbD^{\underline{\varepsilon}^{\prime}},\C)\rightarrow\Map_{n\Cat^{\uple}}(\bbD^{\underline{\varepsilon}^{\prime}},\D)\times_{\Map_{n\Cat^{\uple}}(\partial\bbD^{\underline{\varepsilon}^{\prime}},\D)}\Map_{n\Cat^{\uple}}(\partial\bbD^{\underline{\varepsilon}^{\prime}},\C)
\end{equation*}
for appropriate $\underline{\varepsilon}^{\prime}\in\{0,1\}^n\setminus\{(0,\dotsc,0)\}$. Thus, inducting over $\abs{\underline{\varepsilon}}$ and using left cancellation, we see that $f$ is faithful resp. fully faithful if and only if these maps are monomorphisms resp. equivalences for all $\underline{\varepsilon}^{\prime}\in\{0,1\}^n\setminus\{(0,\dotsc,0)\}$, which can be checked fiberwise over $\Map_{n\Cat^{\uple}}(\partial\bbD^{\underline{\varepsilon}^{\prime}},\C)$.
\end{proof}
\begin{remark}
In particular, a map of $n$-categories is (fully) faithful in our sense if and only if it is (fully) faithful in the usual sense, and a map of anima is (fully) faithful if and only if it is $0$- resp. $(-1)$-truncated.
\end{remark}
\begin{remark}
From this, we deduce by globularity that a (fully) faithful functor of $n$-uple categories induces a (fully) faithful functor on their maximal sub-$n$-categories.
\end{remark}

Inspired by this characterization, we include the following variant.

\begin{definition}
For $0\le k\le n$, a functor $f\colon\C\rightarrow\D$ of $n$-uple Segal anima is \textit{$k$-faithful} if the canonical map
\begin{equation*}
\begin{tikzcd}
\Map_{n\Cat^{\uple}}(\bbD^{\underline{\varepsilon}},\C)\times_{\Map_{n\Cat^{\uple}}(\partial\bbD^{\underline{\varepsilon}},\C)}\{\alpha\}\arrow[r] & \Map_{n\Cat^{\uple}}(\bbD^{\underline{\varepsilon}},\D)\times_{\Map_{n\Cat^{\uple}}(\partial\bbD^{\underline{\varepsilon}},\D)}\{f\alpha\}
\end{tikzcd}
\end{equation*}
is a monomorphism resp. equivalence for every $\underline{\varepsilon}\in\{0,1\}^n\setminus\{(0,\dotsc,0)\}$ such that $\abs{\underline{\varepsilon}}\le k$ resp. $\abs{\underline{\varepsilon}}>k$ and $\alpha\colon\partial\bbD^{\underline{\varepsilon}}\rightarrow\C$. A $1$-faithful functor is also called \textit{locally fully faithful}.
\end{definition}
\begin{remark}
In particular, Proposition \ref{FaithfulCells} yields that $0$- resp. $n$-faithful is the same as (fully) faithful. This notion of a $k$-faithful functor is consistent with \cite[Definition 5.3.1]{Soergel}\footnote{Except that the indexing conventions differ by one.} as well as Definition \ref{LocallyFullyFaithfulDefn}.
\end{remark}

\begin{proposition}\label{CoveringSpaceTheory}
Let $f\colon X\rightarrow Y$ be a faithful map of anima. Then for any map $p\colon Z\rightarrow Y$ of anima, the joint evaluation map
\begin{equation*}
\begin{tikzcd}
\Map_{\An_{/Y}}(Z,X)\arrow[r,hookrightarrow] & \prod_{[z]\in\pi_0(Z)}X_{p(z)}
\end{tikzcd}
\end{equation*}
is an inclusion of discrete anima. The image consists precisely of those families $(x_z)_{[z]\in\pi_0(Z)}$ such that for any $\alpha\colon z\rightarrow z^{\prime}$ in $Z$, the image $p(\alpha)$ is contained in the full sub-anima $\Map_X(x_z,x_{z^{\prime}})\hookrightarrow\Map_Y(p(z),p(z^{\prime}))$.
\end{proposition}
\begin{proof}[Proof.]
There is a cartesian square
\begin{equation*}
\begin{tikzcd}
X\arrow[r]\arrow[d,"f"']\pb & \tau_{\le1}X\arrow[d,"\tau_{\le1}f"']\\
Y\arrow[r] & \tau_{\le1}Y,
\end{tikzcd}
\end{equation*}
immediate by comparing the vertical fibers, which are $0$-truncated in the case of $f$ by assumption. This begets
\begin{equation*}
\Map_{\An_{/Y}}(Z,X)\simeq\Map_{\An_{/\tau_{\le1}Y}}(Z,\tau_{\le_1}X)\simeq\Map_{\An_{/\tau_{\le1}Y}}(\tau_{\le1}Z,\tau_{\le1}X).
\end{equation*}
However, $1$-truncated anima identify with ordinary groupoids, so this claim can be verified by hand.
\end{proof}
\begin{remark}
There is another perspective on this result: if we model anima by CW-complexes, then faithful maps can be modeled by covering spaces and the above is precisely the usual lifting criterion in covering space theory (see e.g. \cite[Theorem 16.1]{Hu}).
\end{remark}
\begin{remark}
In fact, the same argument works with categories instead of anima, replacing the $1$-truncation $\tau_{\le1}X$ with the homotopy category $h\C$. However, we also recover this case below.
\end{remark}

To ultimately prove Theorem \ref{FaithfulLiftingIntro}, we recall an easy lemma about monomorphisms.

\begin{lemma}[\protect{\cite[Proposition A.11]{Ramzi}}]\label{MonosInFun}
Let $\C$ be a finitely complete category and $\I$ a small category. The monomorphisms and the property of factoring through a monomorphism in $\Fun(\I,\C)$ are checked pointwise.\qed
\end{lemma}

\begin{theorem}\label{FaithfulLifting}
Let $f\colon\C\rightarrow\D$ be a faithful functor of $n$-uple categories. Then, for any functor $p\colon\T\rightarrow\D$ of $n$-uple categories, the joint evaluation map
\begin{equation*}
\Map_{n\Cat^{\uple}_{/\D}}(\T,\C)\hookrightarrow\prod_{[t]\in\pi_0(\T^{\simeq})}(\C_{p(t)})^{\simeq}
\end{equation*}
is an inclusion of discrete anima. The image consists precisely of those families $(c_t)_{[t]\in\pi_0(\T^{\simeq})}$ satisfying:
\begin{itemize}
\item For every $\underline{\varepsilon}\in\{0,1\}^n\setminus\{(0,\dotsc,0)\}$ and $\underline{\varepsilon}$-cell $\alpha\colon\bbD^{\underline{\varepsilon}}\rightarrow\T$, its image $f\alpha$ together with the chosen lifts of objects is contained in the full sub-anima $\Map_{n\Cat^{\uple}}(\bbD^{\underline{\varepsilon}},\C)\hookrightarrow\Map_{n\Cat^{\uple}}(\bbD^{\underline{\varepsilon}},\D)\times_{\D^{\simeq,\times2^n}}\C^{\simeq,\times2^n}$.
\end{itemize}
Furthermore, if $1\le k\le n$ and $f\colon\C\rightarrow\D$ is $k$-faithful, it suffices to check the above condition for those $\underline{\varepsilon}\in\{0,1\}^n\setminus\{(0,\dotsc,0)\}$ such that $\abs{\underline{\varepsilon}}=k$. If $\C$ and $\D$ are $(\infty,n)$-categories, it suffices to check the above condition for $\underline{\varepsilon}$ of the form $(1,\dotsc,1,0,\dotsc,0)$.
\end{theorem}
\begin{proof}[Proof.]
This map taking fibers is obtained as a composite
\begin{equation*}
\begin{tikzcd}
\Map_{n\Cat^{\uple}_{/\D}}(\T,\C)\arrow[r,hookrightarrow,"(-)^{\simeq}"] & \Map_{\An_{/\D^{\simeq}}}(\T^{\simeq},\C^{\simeq})\arrow[r,hookrightarrow] & \prod_{[t]\in\pi_0(\T^{\simeq})}(\C_{p(t)})^{\simeq}.
\end{tikzcd}
\end{equation*}
The former is a monomorphism  by Proposition \ref{FaithfulLiftingSegal}. Furthermore, the map $f^{\simeq}\colon\C^{\simeq}\rightarrow\D^{\simeq}$ is faithful, i.e. $0$-truncated, by Corollary \ref{FaithfulOnObjects} and so the latter joint evaluation map is a monomorphism by Proposition \ref{CoveringSpaceTheory}.

To characterize the essential image, note that the condition is clearly necessary. Suppose conversely that $(c_t)_{[t]\in\pi_0(\T^{\simeq})}$ is a tuple satisfying the condition. The tuple is contained in the essential image of the second map; indeed, a morphism $t\rightarrow t^{\prime}$ in $\T^{\simeq}$ is equivalently an invertible $1$-morphism (in, say, the first direction) $t\rightarrow t^{\prime}$ in $\T$ by completeness, whose image under $p$ lifts to an invertible morphism $c_t\rightarrow c_{t^{\prime}}$ in $\C$ by the assumption for $\underline{\varepsilon}=(1,0,\dotsc,0)$ (and since the inverse also lifts), which equivalently describes a lift $c_t\rightarrow c_{t^{\prime}}$ in $\C^{\simeq}$. Thus we obtain an induced map $\T^{\simeq}\rightarrow\C^{\simeq},\,t\mapsto c_t$ over $\D^{\simeq}$ by Proposition \ref{CoveringSpaceTheory}.

To check that this is contained in the essential image of the first map, we use the characterization of Proposition \ref{FaithfulLiftingSegal} and Lemma \ref{MonosInFun}. The Segal condition reduces us to verifying the pointwise factorization at the objects $\underline{\varepsilon}\in\{0,1\}^n$, i.e. we need to check the existence of the dotted maps
\begin{equation*}
\begin{tikzcd}
& \Map_{n\Cat^{\uple}}(\bbD^{\underline{\varepsilon}},\C)\arrow[d,hookrightarrow]\\
\Map_{n\Cat^{\uple}}(\bbD^{\underline{\varepsilon}},\T)\arrow[ur,dotted]\arrow[r] & \Map_{n\Cat^{\uple}}(\bbD^{\underline{\varepsilon}},\D)\times_{\D^{\simeq,\times2^{\abs{\underline{\varepsilon}}}}}\C^{\simeq,\times2^{\abs{\underline{\varepsilon}}}}.
\end{tikzcd}
\end{equation*}
The right vertical map is a monomorphism by faithfulness (and an equivalence if $\underline{\varepsilon}=(0,\dotsc,0)$), so this can be checked on $\pi_0$, where it is precisely the assumption. To verify the final assertions, we induct over the skeletal filtrations of the $\bbD^{\underline{\varepsilon}}$ (cf. the proof of Proposition \ref{FaithfulCells}) to see that it suffices to check the condition for those $\underline{\varepsilon}\in\{0,1\}^n\setminus\{(0,\dotsc,0)\}$ such that $\abs{\underline{\varepsilon}}\le k$ in the case that $f$ is $k$-faithful. Then a retract argument reduces us to the case $\abs{\underline{\varepsilon}}=k$. The other assertion is immediate by globularity.
\end{proof}

\section{Constructing Ambifibrations}\label{AmbifibSection}

In this section, we carry out our main construction of factorization systems, proving Theorem \ref{FactorizationSystemIntro}.

\begin{notation}
The datum of a category $\C$ equipped with a factorization system is denoted $\C^{\dagger}$.
\end{notation}

\begin{definition}\label{AmbiFibDefn}
A functor $p\colon\E^{\dagger}\rightarrow\C^{\dagger}$ in $\Fact$ is called an \textit{ambifibration} if it admits cocartesian resp. cartesian lifts of egressive resp. ingressive morphisms, which are again egressive resp. ingressive. It is called a \textit{curved orthofibration} if furthermore every ingressive morphism in $\E^{\dagger}$ is $p$-cartesian.
\end{definition}

\begin{proposition}[\protect{\cite[Observation 5.2, (2)]{HHNL}}]\label{CocartesianFact}
Let $p\colon\E^{\dagger}\rightarrow\C^{\dagger}$ be a functor in $\Fact$ and $p_{\ing}\colon\E_{\ing}\rightarrow\C_{\ing}$ the induced functor on the wide subcategories of ingressive morphisms. Then an ingressive morphism $f\colon x\rightarrowtail y$ in $\E_{\ing}\hookrightarrow\E$ is $p$-cartesian if and only if it is $p_{\ing}$-cartesian.
\end{proposition}
\begin{proof}[Proof.]
For any egressive morphism $g\colon w\twoheadrightarrow w^{\prime}$ in $\E^{\dagger}$, consider the diagram
\begin{equation*}
\begin{tikzcd}[column sep=0.8em]
& \Map_{\C_{\ing}}(pw^{\prime},px)\arrow[rr,hookrightarrow,color=teal]\arrow[dd,"p(f)_{\circ}",color=teal,near end] & & \Map_{\C}(pw^{\prime},px)\arrow[rr,"p(g)^{\circ}",near start,color=orange]\arrow[dd,"p(f)_{\circ}",rightharpoonup,color=orange,near end,shift left=0.02em]\arrow[dd,rightharpoondown,color=teal,shift right=0.02em] & & \Map_{\C}(pw,px)\arrow[dd,"p(f)_{\circ}",near end,color=orange]\\
\Map_{\E_{\ing}}(w^{\prime},x)\arrow[rr,hookrightarrow,crossing over,color=teal]\arrow[dd,"f_{\circ}",color=teal,near end]\arrow[ur,"p_{\ing}"] & & \Map_{\E}(w^{\prime},x)\arrow[rr,"g^{\circ}",crossing over,near start,color=orange]\arrow[ur,"p"] & & \Map_{\E}(w,x)\arrow[ur,"p"] & \\
& \Map_{\C_{\ing}}(pw^{\prime},py)\arrow[rr,hookrightarrow,color=teal] & & \Map_{\C}(pw^{\prime},py)\arrow[rr,"p(g)^{\circ}",near start,color=orange] & & \Map_{\C}(pw,py).\\
\Map_{\E_{\ing}}(w^{\prime},y)\arrow[rr,hookrightarrow,color=teal,crossing over]\arrow[ur,"p_{\ing}"] & & \Map_{\E}(w^{\prime},y)\arrow[rr,"g^{\circ}",near start,crossing over,color=orange]\arrow[from=uu,"f_{\circ}",rightharpoonup,color=orange,near end,shift left=0.02em]\arrow[from=uu,rightharpoondown,color=teal,shift right=0.02em]\arrow[ur,"p"] & & \Map_{\E}(w,y)\arrow[from=uu,"f_{\circ}",near end,crossing over,color=orange]\arrow[ur,"p"] & \\
\end{tikzcd}
\end{equation*}
The squares consisting of the teal-colored arrows are cartesian, because $f$ and hence $p(f)$ are ingressive and ingressive morphisms satisfy stability under composition and left cancellation. The squares consisting of the orange-colored arrows are cartesian by orthogonality, because $g$ and hence $p(g)$ are egressive. Thus, letting $w\in\E$ be arbitrary and $g=\id_w$, pullback-pasting immediately yields that $f$ is $p_{\ing}$-cartesian if it is $p$-cartesian.

In the converse direction, assume that $f$ is $p_{\ing}$-cartesian and let $g\colon w\rightarrow y$ be an arbitrary morphism in $\E$. Factoring $g$ as an egressive morphism $g^{\prime}\colon w\twoheadrightarrow w^{\prime}$ followed by an ingressive morphism $h\colon w^{\prime}\rightarrowtail y$, we have that $f$ is the image of $h\in\Map_{\E_{\ing}}(w^{\prime},y)$ under $g^{\prime,\circ}$. The left face of the above cuboid is cartesian by assumption, hence the right face of the cuboid also induces an equivalence on vertical fibers over $g$. Since $g$ was arbitrary, this implies the right face of the cuboid is cartesian, i.e. $f$ is $p$-cartesian.
\end{proof}

\begin{corollary}
Let $p\colon\E^{\dagger}\rightarrow\C^{\dagger}$ be a functor in $\Fact$. Then $p$ is an ambifibration resp. curved orthofibration if and only if $p_{\eg}\colon\E_{\eg}\rightarrow\C_{\eg}$ is a cocartesian fibration and $p_{\ing}\colon\E_{\ing}\rightarrow\C_{\ing}$ is a cartesian resp. right fibration.\qed
\end{corollary}
\begin{remark}
In particular, this demonstrates that our definition of a curved orthofibration is consistent with \cite[Definition 5.14]{HHNL} and \cite[Definition 4.1]{Juran} (see also \cite[Lemma 4.2]{Juran}).
\end{remark}

\begin{proposition}\label{FactorizationSystem}
Let $\C^{\dagger}$ be a category equipped with a factorization system and $p\colon\E\rightarrow\C$ a functor admitting cocartesian resp. cartesian lifts of egressive resp. ingressive morphisms. Furthermore, let $\E_c^{\dagger}$ be a factorization system on the fiber $\E_c$ for every $c\in\C$ such that cocartesian resp. cartesian transport along egressive resp. ingressive morphisms preserves egressive resp. ingressive morphisms in the fibers. 

Then there is a factorization system $\E^{\dagger}$ on $\E$ that restricts to the factorization system $\E_c^{\dagger}$ on $\E_c$ for all $c\in\C$ and such that $p^{\dagger}\colon\E^{\dagger}\rightarrow\C^{\dagger}$ is an ambifibration. This factorization system is characterized as follows:
\begin{itemize}
\item The egressive morphisms in $\E^{\dagger}$ are precisely those that factor as a composite of a $p$-cocartesian lift of an egressive morphism in $\C^{\dagger}$ followed by a fiberwise egressive morphism.
\item The ingressive morphisms in $\E^{\dagger}$ are precisely those that factor as a composite of a fiberwise ingressive morphism and a $p$-cartesian lift of an ingressive morphism in $\C^{\dagger}$.
\end{itemize}
\end{proposition}
\begin{proof}[Proof.]
First, let $f\colon x\rightarrow z$ be an arbitrary morphism in $\E$ and factor $p(f)$ as an egressive morphism $g\colon p(x)\twoheadrightarrow y$ followed by an ingressive morphism $h\colon y\rightarrowtail p(z)$. Then consider the diagram
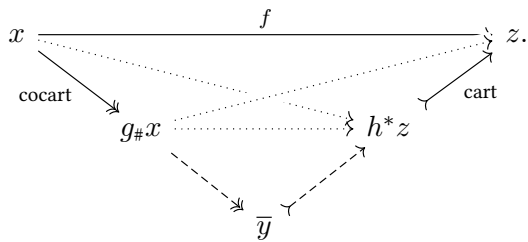
\begin{equation*}
\begin{tikzcd}
x\arrow[rrrr,"f"]\arrow[dr,twoheadrightarrow,"\mathrm{cocart}"']\arrow[drrr,dotted] & & & & z.\\
& g_{\#}x\arrow[rr,dotted]\arrow[urrr,dotted,crossing over]\arrow[dr,twoheadrightarrow,dashed] & & h^{\ast}z\arrow[ur,rightarrowtail,"\mathrm{cart}"'] & \\
& & \overline{y}\arrow[ur,rightarrowtail,dashed] & & 
\end{tikzcd}
\end{equation*}
The dotted morphism $g_{\#}x\rightarrow h^{\ast}z$ is induced from $f$ by the universal properties of (co-)cartesian morphisms and the dashed bottom triangle is the egressive-ingressive factorization of this morphism in the fiber $\E_y^{\dagger}$. This provides an egressive-ingressive factorization of $f$ in $\E^{\dagger}$.

It remains to prove orthogonality. To this end, consider an egressive resp. ingressive in morphism in $\E^{\dagger}$:
\begin{equation*}
\begin{tikzcd}
a\arrow[rr,twoheadrightarrow,"g"]\arrow[dr,twoheadrightarrow,"\mathrm{cocart}"'] & & b, & x\arrow[rr,rightarrowtail,"f"]\arrow[dr,rightarrowtail,"f^{\prime}"'] & & y.\\
& b^{\prime}\arrow[ur,twoheadrightarrow,"g^{\prime}"'] & & & x^{\prime}\arrow[ur,rightarrowtail,"\mathrm{cart}"'] &
\end{tikzcd}
\end{equation*}
Then we form the diagram
\begin{equation*}
\begin{tikzcd}[column sep=0.4em]
& \Map_{\C}(p(b),p(x))\arrow[rr,"p(f^{\prime})_{\circ}",near start,color=violet]\arrow[dd,"p(g^{\prime})^{\circ}",near end,color=violet] & & \Map_{\C}(p(b),p(x^{\prime}))\arrow[rr,color=teal]\arrow[dd,"p(g^{\prime})^{\circ}",near end,color=violet] & & \Map_{\C}(p(b),p(y))\arrow[dd,"p(g^{\prime})^{\circ}",near end,color=violet]\\
\Map_{\E}(b,x)\arrow[rr,"f^{\prime}_{\circ}",near start,crossing over]\arrow[dd,"(g^{\prime})^{\circ}",near end]\arrow[ur,"p"] & & \Map_{\E}(b,x^{\prime})\arrow[rr,crossing over,near start,color=teal]\arrow[ur,"p",color=teal] & & \Map_{\E}(b,y)\arrow[ur,"p",color=teal] & \\
& \Map_{\C}(p(b^{\prime}),p(x))\arrow[rr,"p(f^{\prime})_{\circ}",near start,color=violet]\arrow[dd,color=orange] & & \Map_{\C}(p(b^{\prime}),p(x^{\prime}))\arrow[rr,color=teal]\arrow[dd,color=orange] & & \Map_{\C}(p(b^{\prime}),p(y))\arrow[dd,color=orange]\\
\Map_{\E}(b^{\prime},x)\arrow[rr,"f^{\prime}_{\circ}",near start,crossing over]\arrow[dd,color=orange]\arrow[ur,"p",color=orange] & & \Map_{\E}(b^{\prime},x^{\prime})\arrow[rr,crossing over,color=teal]\arrow[from=uu,"(g^{\prime})^{\circ}",near end,crossing over]\arrow[ur,"p",rightharpoonup,color=teal,shift left=0.02em]\arrow[ur,rightharpoondown,color=orange,shift right=0.02em] & & \Map_{\E}(b^{\prime},y)\arrow[from=uu,"(g^{\prime})^{\circ}",near end,crossing over]\arrow[ur,"p",rightharpoonup,color=teal,shift left=0.02em]\arrow[ur,rightharpoondown,color=orange,shift right=0.02em] & \\
& \Map_{\C}(p(a),p(x))\arrow[rr,"p(f^{\prime})_{\circ}",near start,color=violet] & & \Map_{\C}(p(a),p(x^{\prime}))\arrow[rr,color=teal] & & \Map_{\C}(p(a),p(y)).\\
\Map_{\E}(a,x)\arrow[rr,"f^{\prime}_{\circ}",near start]\arrow[ur,"p",color=orange] & & \Map_{\E}(a,x^{\prime})\arrow[rr,color=teal]\arrow[ur,"p",rightharpoonup,color=teal,shift left=0.02em]\arrow[ur,rightharpoondown,color=orange,shift right=0.02em]\arrow[from=uu,crossing over,color=orange] & & \Map_{\E}(a,y)\arrow[ur,"p",rightharpoonup,color=teal,shift left=0.02em]\arrow[ur,rightharpoondown,color=orange,shift right=0.02em]\arrow[from=uu,crossing over,color=orange] & 
\end{tikzcd}
\end{equation*}
The squares consisting of the orange-colored arrows are cartesian since $a\twoheadrightarrow b^{\prime}$ is $p$-cocartesian, the squares consisting of the teal-colored arrows are cartesian since $x^{\prime}\rightarrowtail y$ is $p$-cartesian and the violet-colored arrows are equivalences. It needs to be proven that the total front face of this diagram is cartesian, which we do by proving that its four component squares are cartesian. For all but the top left square, this reduces to showing the corresponding square on the back face is cartesian by pullback-pasting. For the bottom right back square, this is immediate by orthogonality in $\C^{\dagger}$. For the other two squares, it is clear since their horizontal resp. vertical morphisms are equivalences.

Thus, it remains to prove that top left front square is cartesian. Since the top left back square consists entirely of equivalences, the front square lifts to the slice over $\Map_{\C}(p(b),p(x))$ and we may check that it is cartesian fiberwise. To this end, let $h\colon p(b)\rightarrow p(x)$ be arbitrary and factor it as an egressive morphism $h_{\eg}\colon p(b)\twoheadrightarrow u$ followed by an ingressive morphism $h_{\ing}\colon u\rightarrowtail p(x)$, which admit a cocartesian lift $\tilde{h}_{\eg}\colon b\twoheadrightarrow s$ and a cartesian lift $\tilde{h}_{\ing}\colon t\rightarrowtail x$ respectively. Then we have a natural diagram
\begin{equation*}
\begin{tikzcd}
\Map_{\E_u}(s,t)\arrow[r,"(\tilde{h}_{\eg})^{\circ}"]\arrow[d,"p"]\pb & \Map_{\E}(b,t)\arrow[r,"(\tilde{h}_{\ing})_{\circ}"]\arrow[d,"p"]\pb & \Map_{\E}(b,x)\arrow[d,"p"]\\
\{\ast\}\arrow[r,"h_{\eg}"] & \Map_{\C}(p(b),u)\arrow[r,"(h_{\ing})_{\circ}"] & \Map_{\C}(p(b),p(x)).
\end{tikzcd}
\end{equation*}
The left square is cartesian because $\tilde{h}_{\eg}$ is $p$-cocartesian and the right square is cartesian because $\tilde{h}_{\ing}$ is $p$-cartesian, hence the composite square is also cartesian. In choosing another cocartesian lift $\tilde{h}_{\eg}^{\prime}\colon b^{\prime}\rightarrow s^{\prime}$ of $h_{\eg}$ resp. another cartesian lift $\tilde{h}_{\ing}^{\prime}\colon t^{\prime}\rightarrow x^{\prime}$ of $h_{\ing}$, we obtain morphisms $h_{\eg,\#}g^{\prime}\colon s^{\prime}\rightarrow s$ resp. $h_{\ing}^{\ast}f^{\prime}\colon t\rightarrow t^{\prime}$, which are the cocartesian resp. cartesian transport of $g^{\prime}$ resp. $f^{\prime}$ along $h_{\eg}$ resp. $h_{\ing}$. In total, this allows us to identify the fiber of the top left front square over $f$ with the square
\begin{equation*}
\begin{tikzcd}[column sep=huge]
\Map_{\E_u}(s,t)\arrow[r,"(h_{\ing}^{\ast}f^{\prime})_{\circ}"]\arrow[d,"(h_{\eg,\#}g^{\prime})^{\circ}"] & \Map_{\E_u}(s,t^{\prime})\arrow[d,"(h_{\eg,\#}g^{\prime})^{\circ}"]\\
\Map_{\E_u}(s^{\prime},t)\arrow[r,"(h_{\ing}^{\ast}f^{\prime})_{\circ}"] & \Map_{\E_u}(s^{\prime},t^{\prime}).
\end{tikzcd}
\end{equation*}
This square is cartesian by orthogonality in $\E_u^{\dagger}$, since the assumptions precisely give that $h_{\eg,\#}g^{\prime}\colon s^{\prime}\rightarrow s$ is egressive because $g^{\prime}\colon b^{\prime}\rightarrow b$ is egressive and $h_{\ing}^{\ast}f^{\prime}\colon t\rightarrow t^{\prime}$ is ingressive because $f^{\prime}\colon x\rightarrowtail x^{\prime}$ is ingressive.

Thus, we have proven that $\E^{\dagger}$ is a factorization system and then $p^{\dagger}\colon\E^{\dagger}\rightarrow\C^{\dagger}$ is an ambifibration by definition. Furthermore, because (co-)cartesian lifts of equivalences are equivalences, this factorization system restricts to the given factorization system $\E_c^{\dagger}$ on $\E_c$ for all $c\in\C$.
\end{proof}

\begin{theorem}\label{FactorizationSystemNaturality}
Let $\C^{\dagger}$ be a category equipped with a factorization system and $p\colon\E\rightarrow\C$ a functor, which admits cocartesian resp. cartesian lifts of egressive resp. ingressive morphisms. Then the anima of ambifibrations $p^{\dagger}\colon\E^{\dagger}\rightarrow\C^{\dagger}$ lifting $p$ is discrete and equivalent to the set of collections of factorization systems on the fibers $\E_c,\,c\in\C$ such that cocartesian resp. cartesian transport along egressive resp. ingressive morphisms preserves egressive resp. ingressive morphisms in the fibers.
\end{theorem}
\begin{proof}[Proof.]
Let $p\colon\E^{\dagger}\rightarrow\C^{\dagger}$ be an ambifibration lifting $p$. For any object $c\colon\ast\rightarrow\C$, the factorization system $\E^{\dagger}$ on $\E$ restricts to a factorization system $\E_c^{\dagger}$ on the fiber $\E_c$ by Proposition \ref{CatFactLimits}. Let $g\colon a\twoheadrightarrow b$ be an egressive morphism in $\C$ and $g_{\#}\colon\E_a\rightarrow\E_b$ the associated cocartesian transport. If $f\colon x\twoheadrightarrow y$ is an egressive morphism in $\E_a^{\dagger}$, consider the square
\begin{equation*}
\begin{tikzcd}[column sep=large]
x\arrow[r,twoheadrightarrow,"\mathrm{cocart}"]\arrow[d,twoheadrightarrow,"f"] & g_{\#}x\arrow[d,"g_{\#}f"]\\
y\arrow[r,twoheadrightarrow,"\mathrm{cocart}"] & g_{\#}y.
\end{tikzcd}
\end{equation*}
The cocartesian morphisms are egressive since $p$ is an ambifibration and $f$ is egressive by assumption, hence $g_{\#}f$ is again egressive by stability under composition and right cancellation. The dual argument applies to ingressive morphisms, hence we obtain a map from the anima of ambifibrations lifting $p$ (which is a priori discrete by Proposition \ref{ForgetfulFunctors}) to the set of collections of transport-stable factorization systems on the fibers. 

This map is essentially surjective by Proposition \ref{FactorizationSystem}, but because the classes of egressive/ingressive morphisms given there are necessarily egressive/ingressive for any ambifibration $p\colon\E^{\dagger}\rightarrow\C^{\dagger}$ lifting $p$, this factorization system is unique by mutual orthogonality and the map is also essentially injective.
\end{proof}

\begin{remark}
In the case $\C$ is equipped with the trivial factorization system $\C^{\sharp}$, this recovers \cite[Lemma 3.4]{Shah} as well as (the dual of) \cite[Lemma B.2.4]{Soergel} and is refined in Theorem \ref{UnStraighteningEg}.
\end{remark}
\begin{remark}
In the case that every fiber $\E_c$ is equipped with the trivial factorization system $(\E_c)^{\sharp}$, we obtain the factorization system consisting of lifts of egressive morphisms and cartesian lifts of ingressive morphisms. This is dual to \cite[Proposition 2.1.2.5]{HA} (in this case, the existence of cocartesian lifts of egressive morphisms is an extraneous assumption) and already refined in Juran's un/straightening equivalence Theorem \ref{JuranUnStraightening}.
\end{remark}

\begin{remark}
In the setting of Theorem \ref{FactorizationSystemNaturality}, assume that $p\colon\E\rightarrow\C$ is a bicartesian fibration. Then using the compatibility of orthogonality and adjunctions (e.g. \cite[Lemma 3.1.4]{ABFJ}), the cocartesian transport along a morphism preserves egressive morphisms if and only if the corresponding cartesian transport preserves ingressive morphisms. In particular, the transport-stability is equivalent to requiring that \textit{every} cocartesian (resp. cartesian) transport preserves egressive (resp. ingressive) morphisms in the fibers.
\end{remark}

\begin{example}\label{FernandesExample}
Take the source functor $\ev_0\colon\Fun([1],\An)\rightarrow\An$, which is a bicartesian fibration. Its cartesian transport along $f\colon X\rightarrow Y$ is the functor $f^{\circ}\colon\An_{Y/}\rightarrow\An_{X/}$ over $\An$. Thus, if we equip the fibers with the factorization system obtained by lifting the ($(n+1)$-connected, $(n+1)$-truncated) factorization system \cite[Example 3.1.7, d)]{ABFJ} through the left fibration $p\colon\An_{Z/}\rightarrow\An$ \cite[Lemma 3.1.8]{ABFJ}, these are compatible with cartesian transport along any map in the base. Equipping the base with the ($n$-connected, $n$-truncated) factorization system, Proposition \ref{FactorizationSystem} then induces a factorization system on $\Fun([1],\An)$. This recovers the factorization system of \cite[Definition 7.1.1]{Fernandes}, which inspired our construction.
\end{example}

\section{The Un/Straightening Equivalence}\label{UnStraighteningSection}

In this section, we prove Theorem \ref{UnStraighteningIntro} and other un/straightening equivalences for categories equipped with a factorization system such as Theorem \ref{UnStraighteningEg}. This is done by lifting $2$- or double-categorical un/straightening equivalences by combining Theorem \ref{FaithfulLifting} and the classification of Theorem \ref{FactorizationSystemNaturality}.

First, we need to define the $2$- and double categories of categories equipped with a factorization system.

\begin{definition}
The double category $\bbFact$ of categories equipped with a factorization system is defined as the wide and locally full sub-double Segal anima $\bbFact\hookrightarrow\bbCat\times_{T\Cat^{\simeq}}T\Fact^{\simeq}$ on the horizontal resp. vertical morphisms corresponding to functors preserving egressive resp. ingressive morphisms.
\end{definition}
\begin{remark}
The above pullback is formed along the forgetful functor $Tp^{\simeq}\colon T\Fact^{\simeq}\rightarrow T\Cat^{\simeq}$ and the functor $\bbCat\rightarrow T\Cat^{\simeq}$ adjoint to the equivalence $\bbCat_{0,0}\simeq\Map_{2\Cat}([0]\boxtimes[0],\Cat)\simeq\Cat^{\simeq}$.
\end{remark}

\begin{proposition}\label{ForgetDoubleFunctor}
The forgetful functor $p\colon\bbFact\rightarrow\bbCat$ is locally fully faithful. For a category $\C$, the core $p^{-1}(\C)^{\simeq}$ of the fiber is discrete and equivalent to the set of factorization systems on $\C$.
\end{proposition}
\begin{proof}[Proof.]
The functor $p$ is locally fully faithful by definition. The fiber $\bbFact\times_{\bbCat}\{\C\}$ is a wide, locally full sub-double Segal anima of $T\Fact^{\simeq}\times_{T\Cat^{\simeq}}\{\C\}$ and $T$ preserves limits, hence we conclude by Proposition \ref{ForgetfulFunctors}.
\end{proof}

\begin{proposition}
The double Segal anima $\bbFact$ is in fact complete, i.e. a double category.
\end{proposition}
\begin{proof}[Proof.]
We prove that $\bbFact$ is vertically complete. The horizontal completeness is proven analogously. To this end, consider the commutative diagram
\begin{equation*}
\begin{tikzcd}
\bbFact_{1,0}\arrow[d,"s_0"]\arrow[r,hookrightarrow] & (\bbCat\times_{T\Cat^{\simeq}}T\Fact^{\simeq})_{1,0}\arrow[d,"s_0"]\arrow[r,symbol=\simeq] & \bbCat_{1,0}\times_{\Cat^{\simeq,\times2}}\Fact^{\simeq,\times2}\arrow[d,"s_0"]\\
\bbFact_{1,\eq}\arrow[r,hookrightarrow] & (\bbCat\times_{T\Cat^{\simeq}}T\Fact^{\simeq})_{1,\eq}\arrow[r,symbol=\simeq] & \bbCat_{1,\eq}\times_{\Cat^{\simeq,\times4}}\Fact^{\simeq,\times4}.
\end{tikzcd}
\end{equation*}
The degeneracy $s_0\colon\bbCat_{1,0}\rightarrow\bbCat_{1,\eq}$ is an equivalence since $\bbCat$ is a double category, hence the right vertical map is a base-change of $\Fact^{\simeq,\times2}\rightarrow\Cat^{\simeq,\times2}\times_{\Cat^{\simeq,\times4}}\Fact^{\simeq,\times4}$ by pullback-pasting and thus a monomorphism since $p^{\simeq}\colon\Fact^{\simeq}\rightarrow\Cat^{\simeq}$ is $0$-truncated by Propositions \ref{ForgetfulFunctors} and \ref{FaithfulIsTruncatedOnCore}.

Hence the left vertical map $s_0\colon\bbFact_{1,0}\rightarrow\bbFact_{1,\eq}$ is a monomorphism by left cancellation and it remains to verify that it is essentially surjective, i.e. that every vertically invertible square in $\bbFact$ is in fact vertically degenerate. The underlying square in $\bbCat$ is vertically degenerate, so we conclude by Proposition \ref{EgIsoFull}.
\end{proof}
\begin{remark}
Informally speaking, the double category $\bbFact$ has as objects categories equipped with a factorization system, as horizontal resp. vertical morphisms the functors preserving egressive resp. ingressive morphisms and as squares the appropriate oplax commutative squares in the $2$-category $\bfCat$.
\end{remark}

\begin{definition}
The horizontal resp. vertical fragment of $\bbFact$ are denoted $\bfFact^{\eg}\coloneqq\Hor(\bbFact)$ resp. $\bfFact^{\ing}\coloneqq\Ver(\bbFact)$. The $2$-category $\bfFact$ is the wide and locally full subcategory of $\bfFact^{\eg}$ (or, equivalently, $\bfFact^{\ing}$) on the functors preserving both egressive and ingressive morphisms.

There are locally fully faithful functors $p\colon\bfFact\rightarrow\bfCat$ and $p^{\eg}\colon\bfFact^{\eg}\rightarrow\bfCat$ resp. $\bfFact^{\ing}\rightarrow\bfCat$, which on underlying $1$-categories recover the faithful functors $p\colon\Fact\rightarrow\Cat$ and $p^{\eg}\colon\Fact^{\eg}\rightarrow\Cat$ resp. $p^{\ing}\colon\Fact^{\ing}\rightarrow\Cat$ by Proposition \ref{ForgetfulFunctors}.
\end{definition}

\begin{definition}
Let $\C^{\dagger}$ be a category equipped with a factorization system. The full subcategory of $\bfFact^{\eg}_{/\C^{\dagger}}$ on the ambifibrations is denoted by $\bfAmbi^{\eg,\lax}(\C^{\dagger})$ and the wide, locally full subcategory thereof on the functors over $\C^{\dagger}$ preserving cocartesian egressive morphisms as well as cartesian ingressive morphisms by $\bfAmbi^{\eg}(\C^{\dagger})$. Its intersection with $\bfFact_{/\C^{\dagger}}$ is denoted by $\bfAmbi(\C^{\dagger})$ and the underlying $1$-category thereof by $\Ambi(\C^{\dagger})$. Finally, the full subcategory of $\Ambi(\C^{\dagger})$ on the curved orthofibrations is denoted by $\Ortho(\C^{\dagger})$.
\end{definition}
\begin{remark}
Proposition \ref{EgIsoFull} implies that $\bfAmbi^{\eg}(\C^{\dagger})^{\simeq}\simeq\bfAmbi(\C^{\dagger})^{\simeq}$. There are also dual versions denoted $\bfAmbi^{\ing,\oplax}(\C^{\dagger})$, etc.
\end{remark}

\begin{proposition}\label{ForgetOrtho}
Let $\C^{\dagger}$ be a category equipped with a factorization system. The forgetful map $\Ortho(\C^{\dagger})^{\simeq}\rightarrow(\Cat_{/\C})^{\simeq}$ is a monomorphism whose essential image consists of those functors $p\colon\E\rightarrow\C$ that admit cocartesian resp. cartesian lifts of egressive resp. ingressive morphisms.
\end{proposition}
\begin{proof}[Proof.]
The functor $p\colon\Fact\rightarrow\Cat$ is faithful by Proposition \ref{ForgetfulFunctors}, hence so is $\Fact_{/\C}\rightarrow\Cat_{/\C}$ and thus $(\Fact_{/\C})^{\simeq}\rightarrow(\Cat_{/\C})^{\simeq}$ is $0$-truncated by Proposition \ref{FaithfulIsTruncatedOnCore}. Restricting the source to the full sub-anima $\Ortho(\C)^{\simeq}$, we observe that there is at most one factorization system $\E^{\dagger}$ on $\E$ such that $p\colon\E^{\dagger}\rightarrow\C^{\dagger}$ is a curved orthofibration, because the ingressive morphisms in $\E^{\dagger}$ are precisely the cartesian lifts of ingressive morphisms in $\C^{\dagger}$ and this determines the factorization system. Thus, the forgetful map is $(-1)$-truncated, i.e. a monomorphism, and the characterization of the essential image follows by Proposition \ref{FactorizationSystem}.
\end{proof}

\begin{theorem}[Un/Straightening for Factorization Systems]\label{UnStraightening}
Let $\C^{\dagger}$ be a category equipped with a factorization system. Then there is a natural equivalence
\begin{equation*}
\Ambi(\C^{\dagger})^{\simeq}\simeq\Map_{\DCat}(\iota(\C^{\dagger})^{v-\op},\bbFact).
\end{equation*}
\end{theorem}
\begin{proof}[Proof.]
To prove this theorem, we consider the square
\begin{equation}\label{UnStraighteningSquare}
\begin{tikzcd}
\Map_{\DCat}(\iota(\C^{\dagger})^{v-\op},\bbFact)\arrow[r,dotted,"\simeq"]\arrow[d] & \Ambi(\C^{\dagger})^{\simeq}\arrow[d]\\
\Map_{\DCat}(\iota(\C^{\dagger})^{v-\op},\bbCat)\arrow[r,"\simeq"] & \Ortho(\C^{\dagger})^{\simeq}.
\end{tikzcd}
\end{equation}
The bottom un/straightening equivalence is Juran's Theorem \ref{JuranUnStraightening}. The left vertical map is induced by the (locally fully) faithful functor $p\colon\bbFact\rightarrow\bbCat$ and hence $0$-truncated by Corollary \ref{FaithfulIsTruncated}. The right vertical map is induced by the diagram
\begin{equation*}
\begin{tikzcd}
& \Ortho(\C^{\dagger})^{\simeq}\arrow[d,hookrightarrow]\\
\Ambi(\C^{\dagger})^{\simeq}\arrow[ur,dotted]\arrow[r] & (\Cat_{/\C})^{\simeq},
\end{tikzcd}
\end{equation*}
in which the right vertical map is a monomorphism by Proposition \ref{ForgetOrtho} and the dotted lifting exists by the characterization of the essential image. This map is also $0$-truncated by the argument of Proposition \ref{ForgetOrtho}. To construct the dotted map in $(\ref{UnStraighteningSquare})$, we can thus use Proposition \ref{CoveringSpaceTheory}. The lifts on objects are specified by Proposition \ref{FactorizationSystem} and satisfy the naturality condition by construction.

To then prove the dotted map is an equivalence, we can check that the square $(\ref{UnStraighteningSquare})$ induces equivalences on vertical fibers. The left-hand fiber is described by Theorem \ref{FaithfulLifting} in light of Proposition \ref{ForgetDoubleFunctor} as the set of collections of transport-stable factorization systems on the fibers of the corresponding fibration. The induced equivalence is then precisely the content of Theorem \ref{FactorizationSystemNaturality}.
\end{proof}

\begin{remark}
This un/straightening equivalence projects onto Juran's un/straightening equivalence Theorem \ref{JuranUnStraightening} by construction. This projection also admits a section given by the inclusion $\Ortho(\C^{\dagger})^{\simeq}\hookrightarrow\Ambi(C^{\dagger})^{\simeq}$ resp. the section $(-)^{\sharp}\colon\bbCat\hookrightarrow\bbFact$ induced by the diagram
\begin{equation*}
\begin{tikzcd}
\bbCat\arrow[r,dotted,"(-)^{\sharp}"]\arrow[d] & \bbCat\times_{T\Cat^{\simeq}}T\Fact^{\simeq}\arrow[r]\arrow[d]\pb & \bbCat\arrow[d]\\
T\Cat^{\simeq}\arrow[r,"(-)^{\sharp}"] & T\Fact^{\simeq}\arrow[r] & T\Cat^{\simeq},
\end{tikzcd}
\end{equation*}
in which the horizontal composite is the identity. This indeed factors through $\bbFact\hookrightarrow\bbCat\times_{T\Cat^{\simeq}}T\Fact^{\simeq}$ since any functor of categories preserves the trivial factorization system $(-)^{\sharp}$.
\end{remark}

In the case where $\C$ is equipped with a trivial factorization system, we obtain a stronger un/straightening equivalence, which also encodes lax phenomena. To state this concisely, we use the formalism of decorated $2$-categories \cite[Sections 2.6-2.7]{AHM}.

\begin{notation}
A \textit{decorated} $2$-category $\bm{\C}^{\diamond}$ consists of an underlying $2$-category $\bm{\C}$ as well as collections of $1$- and $2$-morphisms in $\bm{\C}$, either stable under composition and containing the invertible ones. A decorated functor is a functor of $2$-categories preserving the decorated $1$- and $2$-morphisms respectively.

For two decorated $2$-categories $\bm{\C}^{\diamond}$ and $\bm{\D}^{\diamond}$, we denote by $\bfFun^{(\op)\lax}(\bm{\C}^{\diamond},\bm{\D}^{\diamond})^{\diamond}$ the locally full subcategory of $\bfFun^{(\op)\lax}(\bm{\C},\bm{\D})$ on the decorated functors $\bm{\C}^{\diamond}\rightarrow\bm{\D}^{\diamond}$ and those (op)lax natural transformations whose (op)lax naturality squares at any decorated $1$-morphism in $\bm{\C}^{\diamond}$ are filled by a decorated $2$-morphism in $\bm{\D}^{\diamond}$. This category is decorated by those (op)lax natural transformations all of whose (op)lax naturality squares are filled by a decorated $2$-morphism in $\bm{\D}^{\diamond}$ and those modifications all of whose $2$-morphism components are decorated in $\bm{\D}^{\diamond}$.
\end{notation}

\begin{definition}
The $2$-category $\bfFact^{\eg}$ is upgraded to a decorated $2$-category $\bfFact^{\eg,\diamond}$ by regarding all $1$-morphisms as well as those $2$-morphisms corresponding to natural transformations all of whose components are egressive morphisms as decorated. This induces in particular a decoration $\bfAmbi^{\eg,\lax}(\C^{\dagger})^{\diamond}$ on $\bfAmbi^{\eg,\lax}(\C^{\dagger})$ for any category $\C^{\dagger}$ equipped with a factorization system.

The maximal subcategories on all the decorated $1$- and $2$-morphisms, which are thus wide and locally wide, are denoted $\bfFact^{\eg,\eg}$ resp. $\bfAmbi^{\eg,\eg,\lax}(\C^{\dagger})$.
\end{definition}

\begin{remark}
The $2$-categories $\bfFact^{\eg}$ resp. $\bfFact^{\eg,\eg}$ are both tensored over $\Cat$, by $(\C,\D^{\dagger})\mapsto\C^{\flat}\times\D^{\dagger}$ resp. $(\C,\D^{\dagger})\mapsto\C^{\sharp}\times\D^{\dagger}$.
\end{remark}

\begin{theorem}\label{UnStraighteningEg}
Let $\C$ be a category. Then there is a natural equivalence of decorated $2$-categories
\begin{equation*}
\bfAmbi^{\eg,\lax}(\C^{\sharp})^{\diamond}\simeq\bfFun^{\lax}\br{\C^{\sharp},\bfFact^{\eg,\diamond}}^{\diamond}.
\end{equation*}
\end{theorem}
\begin{remark}
By Proposition \ref{CocartesianFact}, the objects of $\bfAmbi^{\eg,\lax}(\C^{\sharp})^{\diamond}$ are equivalently cocartesian fibrations $\E\rightarrow\C$ in $\Cat$ together with a factorization system on $\E$ in which every $p$-cocartesian morphism is egressive.
\end{remark}
\begin{proof}
To this end, consider the lifting problem
\begin{equation*}
\begin{tikzcd}
& & \bfFact^{\eg}\arrow[d,"p^{\eg}"']\\
\C\boxtimes\bfAmbi^{\eg,\lax}(\C^{\sharp})\arrow[r]\arrow[urr,dotted] & \C\boxtimes\bfCocart^{\lax}(\C)\arrow[r] & \bfCat.
\end{tikzcd}
\end{equation*}
The second horizontal functor is curried from the un/straightening equivalence \cite[Theorem E]{HHNL}. To produce the dotted lift, we can use Theorem \ref{FaithfulLifting} as $p^{\eg}\colon\bfFact^{\eg}\rightarrow\bfCat$ is locally fully faithful. Thus, we need to specify lifts of the objects and verify a compatibility condition for $1$-morphisms. To do so, we recall the description of the Gray tensor product of Proposition \ref{GrayRudimentary}. For an object $c\in\C$ and an ambifibration $p\colon\E^{\dagger}\rightarrow\C^{\sharp}$, the horizontal composite maps $c\boxtimes p\mapsto\E_c$ and we specify its lift as the factorization system $\E_c^{\dagger}$ obtained from the factorization system $\E^{\dagger}$ by restriction (cf. Proposition \ref{CatFactLimits}).

For any morphism $f\colon c\rightarrow d$ in $\C$ and ambifibration $p\colon\E^{\dagger}\rightarrow\C^{\sharp}$, the horizontal composite maps $f\boxtimes p$ to the cocartesian transport $f_{\#}\colon\E_c\rightarrow\E_d$, which lifts to a functor $f_{\#}\colon\E_c^{\dagger}\rightarrow\E_d^{\dagger}$ preserving egressive morphisms by the argument of Theorem \ref{FactorizationSystemNaturality}. For any object $c\in\C$ and functor $\alpha\colon\E^{\dagger}\rightarrow\F^{\dagger}$ over $\C^{\sharp}$ preserving egressive morphisms, the horizontal composite maps $c\boxtimes\alpha$ to the restricted functor $\alpha_c\colon\E_c^{\dagger}\rightarrow\F_c^{\dagger}$, which clearly preserves egressive morphisms. This produces the desired lift.

Then, currying back, we obtain the upper horizontal functor in the diagram
\begin{equation}\label{UnStraighteningLaxSquare}
\begin{tikzcd}
\bfAmbi^{\eg,\lax}(\C^{\sharp})\arrow[r]\arrow[d] & \bfFun^{\lax}\br{\C,\bfFact^{\eg}}\arrow[d]\\
\bfCocart^{\lax}(\C)\arrow[r,"\sim"] & \bfFun^{\lax}\br{\C,\bfCat}.
\end{tikzcd}
\end{equation}
The left vertical functor is obtained by restricting the locally fully faithful functor $\bfFact^{\eg}_{/\C^{\sharp}}\rightarrow\bfCat_{/\C}$ to full subcategories, hence it is locally fully faithful. The right vertical functor is locally fully faithful since $p^{\eg}\colon\bfFact^{\eg}\rightarrow\bfCat$ is locally fully faithful (e.g. by \cite[Theorem 2.7.5]{AGH}). By left cancellation, the top horizontal functor is locally fully faithful. It is also essentially surjective by Proposition \ref{FactorizationSystem}.

Furthermore, let $F,G\colon\C\rightarrow\bfFact^{\eg}$ be two functors and $\alpha\colon F\Rightarrow G$ a natural transformation. Then the functor $\Un(\alpha)\colon\Un(F)\rightarrow\Un(G)$ over $\C$ obtained by unstraightening the underlying natural transformation upgrades to a morphism in $\bfFact^{\eg}$ if and only if it maps every $p$-cocartesian morphism in $\Un(F)$ to an egressive morphism in $\Un(G)$ as it already preserves fiberwise egressive morphisms by construction. For a $p$-cocartesian morphism $f\colon x\rightarrow p(f)_{\#}x$ in $\Un(F)$, the canonical map $p(f)_{\#}\Un(\alpha)x\rightarrow\Un(\alpha)p(f)_{\#}x$ is precisely the laxness constraint of $\alpha$ at $f$, so right cancellation implies that $\Un(\alpha)f\colon\Un(\alpha)x\rightarrow\Un(\alpha)p(f)_{\#}x$ is egressive if and only this laxness constraint is egressive.

In total, this implies that the upper horizontal functor in $(\ref{UnStraighteningLaxSquare})$ is an equivalence onto the wide and locally full subcategory of $\bfFun^{\lax}(\C,\bfFact^{\eg})$ that underlies $\bfFun^{\lax}\br{\C^{\sharp},\bfFact^{\eg,\diamond}}^{\diamond}$. It remains to verify that this equivalence identifies the corresponding decorations, which is vacuous on $1$-morphisms and a consequence of how the un/straightening equivalence acts on modifications on $2$-morphisms.
\end{proof}
\begin{remark}
There is, of course, a dual version $\bfAmbi^{\ing,\oplax}(\C^{\flat})^{\diamond}\simeq\bfFun^{\oplax}\br{\C^{\sharp},\bfFact^{\ing,\diamond}}^{\diamond}$ (note that one marking changes and the other does not).
\end{remark}

\begin{corollary}\label{UnStraighteningEgNonLax}
Let $\C$ be a category. Then there is a natural equivalence of $2$-categories
\begin{equation*}
\bfAmbi^{\eg}(\C^{\sharp})\simeq\bfFun(\C,\bfFact^{\eg}).
\end{equation*}
\end{corollary}
\begin{proof}[Proof.]
This follows from Theorem \ref{UnStraighteningEg} by pulling back to the (non-lax) un/straightening equivalence
\begin{equation*}
\bfCocart(\C)\iso\bfFun(\C,\bfCat).
\end{equation*}
Indeed, $\bfFun^{\lax}(\C,\bfFact^{\eg})$ pulls back to $\bfFun(\C,\bfFact^{\eg})$ since the forgetful functor $p^{\eg}\colon\bfFact^{\eg}\rightarrow\bfCat$ is locally fully faithful and thus conservative on $2$-cells and $\bfAmbi^{\eg,\lax}(\C^{\sharp})$ pulls back to $\bfAmbi^{\eg}(\C^{\sharp})$ by definition.
\end{proof}

\begin{corollary}
Let $\C$ be a category. Then there is a natural equivalence of $2$-categories
\begin{equation*}
\bfAmbi^{\eg,\eg,\lax}(\C^{\sharp})\simeq\bfFun^{\lax}(\C,\bfFact^{\eg,\eg}).
\end{equation*}
\end{corollary}
\begin{proof}[Proof.]
This follows from Theorem \ref{UnStraighteningEg} by passing to the maximal subcategories on all the decorated $1$- and $2$-morphisms.
\end{proof}

\begin{remark}
The general paradigm is that any categorical un/straightening equivalence can be refined to an un/straightening equivalence for factorization systems by using this fairly robust proof strategy. The un/straightening equivalence for double categories \cite[Theorem B]{Juran} is proven at the level of anima, but it should be expected that this arises from an un/straightening equivalence at the level of double categories itself. The above paradigm would then upgrade Theorem \ref{UnStraightening} to an equivalence of double categories
\begin{equation*}
\bbAmbi(\C^{\dagger})\simeq\bbFun(\iota(\C^{\dagger})^{v-\op},\bbFact),
\end{equation*}
where $\bbFun$ denotes the internal $\Hom$ in $\DCat$ and $\bbAmbi(\C)$ is the double category such that:
\begin{itemize}
\item The objects are ambifibration $p\colon\E^{\dagger}\rightarrow\C^{\dagger}$ in $\Fact$.
\item The horizontal resp. vertical morphisms are functors over $\C$ that preserve cocartesian egressive morphisms and cartesian ingressive morphisms as well as arbitrary egressive resp. ingressive morphisms.
\item The squares are appropriate oplax commutative squares over $\C$.
\end{itemize}
We note that this prediction is consistent with Corollary \ref{UnStraighteningEgNonLax}. Indeed, we have $\iota(\C^{\sharp})=\C_h$ and $(-)_h\simeq(-)^{v-\op}\circ(-)_h$, so that $\iota(\C^{\sharp})^{v-\op}\simeq\C_h$ and
\begin{equation*}
\Hor(\bbAmbi(\C^{\sharp}))\simeq\bfAmbi(\C^{\sharp})\simeq\bfFun(\C,\bfFact^{\eg})\simeq\bfFun(\C,\Hor(\bbFact))\simeq\Hor(\bbFun(\C_h,\bbFact)).
\end{equation*}
\end{remark}

\section{Monadicity}\label{MonadicitySection}

In this section, we prove Theorem \ref{MonadicityIntro} that the forgetful functor $p\colon\Fact\rightarrow\Cat$ is monadic, extending a result of Korostenski-Tholen \cite[Theorem B]{KorostenskiTholen}. This result has not yet appeared in the literature to our knowledge, though it can be deduced in a roundabout fashion. Namely, Kositsyn \cite{Kositsyn} exhibits a category $\Delta_{\Fact}^{\op}$ such that \singlequote{Segal objects}\footnote{Though we stress that these are \textit{not} Segal objects in the sense of an algebraic pattern \cite{ChuHaugseng}.} $\Delta_{\Fact}^{\op}\rightarrow\C$ are equivalent to ordinary Segal objects in $\C$ equipped with a factorization system, and such that the forgetful functor is induced by a map $t\colon\Delta^{\op}\rightarrow\Delta_{\Fact}^{\op}$. This can be used to check that the free-forgetful adjunction is monadic by an argument along the lines of \cite[Proposition 8.1]{ChuHaugseng}, which then restricts to subcategories of complete objects. Instead, we will present a direct proof based on an application of the Barr-Beck-Lurie monadicity theorem.

The following lemma appears to be folklore.

\begin{lemma}\label{CoinitialLeft}
Let $f\colon\I\rightarrow\J$ be a coinitial functor and $p\colon\E\rightarrow\B$ a left fibration. Then the square
\begin{equation*}
\begin{tikzcd}
\Fun(\J,\E)\arrow[r,"p_{\circ}"]\arrow[d,"f^{\circ}"]\pb & \Fun(\J,\B)\arrow[d,"f^{\circ}"]\\
\Fun(\I,\E)\arrow[r,"p_{\circ}"] & \Fun(\I,\B)
\end{tikzcd}
\end{equation*}
is cartesian.
\end{lemma}
\begin{proof}[Proof.]
This can be checked by applying the functors $\Map_{\Cat}([n],-)\colon\Cat\rightarrow\An$, which are jointly conservative and limit-preserving. Then, since $f\times[n]\colon\I\times[n]\rightarrow\J\times[n]$ is again coinitial, currying reduces this to the underlying square of anima, which is cartesian by (the dual of) \cite[Theorem 6.5.13]{Haugseng}.
\end{proof}

In the following, we treat split simplicial objects according to the notation of \cite[Subsection 10.1.6]{kerodon}.

\begin{lemma}\label{SplitPointedSimplicialAnima}
Let $X_{\bullet}\colon\Delta^{\op}\rightarrow\An$ be a simplicial anima. Then:
\begin{enumerate}[label=(\arabic*)]
\item A pointing $x\in X_0$ naturally determines a pointed simplicial anima $(X_{\bullet},x_{\bullet})\colon\Delta^{\op}\rightarrow\An_{\ast/}$ lifting $X_{\bullet}$.
\item In the setting of (1), if $\overline{X}_{\bullet}\colon\Delta_+^{\op}\rightarrow\An$ is an augmentation extending $X_{\bullet}$, this naturally lifts to an augmentation $(\overline{X}_{\bullet},x_{\bullet})\colon\Delta_+^{\op}\rightarrow\An_{\ast/}$ extending $(X_{\bullet},x_{\bullet})$. Here, $X_{-1}$ is pointed by $x_{-1}\coloneqq d^{\ast}(x)$, where $d\colon[-1]\rightarrow[0]$ is the unique map in $\Delta_+$.
\item In the setting of (2), if $\overline{X}_{\bullet}$ is split and $x=s^{\ast}(x^{\prime})$ for some $x^{\prime}\in X_{-1}$, where $s\colon[0]\rightarrow[-1]$ is the unique extra degeneracy, then $(\overline{X}_{\bullet},x_{\bullet})$ is again split.
\end{enumerate}
\end{lemma}
\begin{proof}[Proof.]
The claims (1) and (2) are consequences of Lemma \ref{CoinitialLeft}, since $\{[0]\}^{\op}\hookrightarrow\Delta^{\op}$ resp. $\Delta^{\op}\hookrightarrow\Delta_+^{\op}$ are coinitial. To prove (3), pick a splitting $\Delta_+^{\op}\rightarrow\Delta_{\min}^{\op}\rightarrow\An$ of $\overline{X}_{\bullet}$ and consider the diagram
\begin{equation*}
\begin{tikzcd}
\{[0]\}^{\op}\arrow[r]\arrow[d] & \Delta_{\min,\le1}^{\op}\arrow[r]\arrow[d] & \An_{\ast/}\arrow[d]\\
\Delta_+^{\op}\arrow[r] & \Delta_{\min}^{\op}\arrow[r]\arrow[ur,dotted] & \An.
\end{tikzcd}
\end{equation*}
The top right horizontal functor is induced by pointing $X_{-1}$ at $x^{\prime}$ and $X_0$ at $x\coloneqq s^{\ast}(x^{\prime})$. The indicated lifting problem admits a solution as $\Delta_{\min,\le1}^{\op}\hookrightarrow\Delta_{\min}^{\op}$ is coinitial (see below), so that the composite lift determines the augmented pointed simplicial anima $(\overline{X}_{\bullet},x_{\bullet})$ by uniqueness and is split by construction.

To verify coinitiality, we check Quillen's Theorem A: for any $[n]\in\Delta_{\min}$, the relative slice category $\Delta_{\min,\le1}\times_{\Delta_{\min}}(\Delta_{\min})_{[n]/}$ is equivalent to the join $\{1,\dotsc,n\}\ast\Delta_{\min,\le1}$, hence weakly contractible.
\end{proof}

\begin{proposition}\label{FullyFaithfulSectionSplitColimit}
Let $F_{\bullet}\colon\C_{\bullet}\hookrightarrow\D_{\bullet}$ be a natural transformation of split simplicial objects in $\Cat$ that admits a retraction and is levelwise fully faithful. Then the induced functor on colimits $|F_{\bullet}|\colon|\C_{\bullet}|\rightarrow|\D_{\bullet}|$ admits a retraction and is fully faithful.
\end{proposition}
\begin{proof}[Proof.]
Extend $F_{\bullet}$ to a natural transformations of augmented simplicial objects by taking colimits. These augmented simplicial objects are split by assumption, in particular absolute colimits. We need to prove that $F_{-1}\colon\C_{-1}\rightarrow\D_{-1}$ is fully faithful, which can be checked by applying the functors $\Map_{\Cat}([n],-)\colon\Cat\rightarrow\An$, which jointly detect fully faithfulness. Thus, we reduce to the case of (split) augmented simplicial anima.

In this case, we have to check that $F_{-1}$ induces an injection on $\pi_0$ and bijections on $\pi_n,\,n\ge1$ for all choices of basepoints. The claim about $\pi_0$ is immediate since this is a map in $\Set$ that admits a retraction. To verify the claim about $\pi_n$, pick a basepoint $y\in\C_{-1}$ and set $F_{-1}(y)\in\D_{-1}$, so that (2) of Lemma \ref{SplitPointedSimplicialAnima} lifts $F_{\bullet}\colon\C_{\bullet}\rightarrow\D_{\bullet}$ to a natural transformation of augmented pointed simplicial anima, which are again split by (3) of Lemma \ref{SplitPointedSimplicialAnima}. These absolute colimits are thus preserved by the functors $\pi_n\colon\An_{\ast/}\rightarrow\Ab,\,n\ge1$.
\end{proof}
\begin{remark}
The retraction of $F_{\bullet}$ was only needed to get injectivity on $\pi_0$, but this additional assumption is not redundant, e.g. given any non-constant split simplicial set $X_{\bullet}$, the counit map $\const_{X_0}\rightarrow X_{\bullet}$ is a levelwise injection of split simplicial sets, but induces the non-injective map $X_0\twoheadrightarrow X_{-1}$ on colimits.
\end{remark}

\begin{proposition}\label{FreeForgetful}
There is an adjunction $\Ar(-)^{\dagger}\colon\Cat\rightleftarrows\Fact\colon p$. Here, $\Ar(\C)$ is equipped with the canonical factorization system where the egressive morphisms are those projecting to equivalences on the source and the ingressive morphisms are those projecting to equivalences on the target.
\end{proposition}
\begin{proof}[\protect{Proof (\cite[Proposition 4.8]{HHNL})}]
The unit transformation $\eta\colon\C\rightarrow\Ar(\C)$ is induced by pulling back along $[1]\rightarrow\ast$. The counit transformation $\varepsilon\colon\Ar(\C)^{\dagger}\rightarrow\C^{\dagger}$ is given by the composite $\Ar(\C)\stackrel{\mathrm{fact}}{\hookrightarrow}\Fun([2],\C)\stackrel{\ev_1}{\rightarrow}\C$, which one checks to be a morphism in $\Fact$. The triangle identities may then be verified by hand.
\end{proof}

\begin{theorem}\label{MonadicityTheorem}
The forgetful functor $p\colon\Fact\rightarrow\Cat$ is monadic.
\end{theorem}
\begin{proof}[Proof.]
The forgetful functor $p\colon\Fact\rightarrow\Cat$ is a right adjoint by Proposition \ref{FreeForgetful} and conservative by Proposition \ref{ForgetfulFunctors}. To check the conditions of the monadicity theorem \cite[Theorem 4.7.3.5]{HA}, it remains to prove that any $p$-split simplicial object $\C_{\bullet}^{\dagger}\colon\Delta^{\op}\rightarrow\Fact$ admits a colimit in $\Fact$ that is preserved by $p$. To this end, consider a splitting $\C_{\bullet}\colon\Delta_+^{\op}\rightarrow\Delta_{\min}^{\op}\rightarrow\Cat$ of its image $p\C_{\bullet}^{\dagger}$.

The natural transformation $\mathrm{fact}\colon\Fun([1],\C_{\bullet})\hookrightarrow\Fun([2],\C_{\bullet})$ of split simplicial objects given by the functorial egressive-ingressive factorizations is levelwise fully faithful and admits a retraction given by $d_1^{\ast}$, hence Proposition \ref{FullyFaithfulSectionSplitColimit} implies that the induced map on colimits $\mathrm{fact}\colon\Fun([1],\C_{-1})\hookrightarrow\Fun([2],\C_{-1})$ is fully faithful and admits a retraction given by $d_1^{\ast}$. Furthermore, this fits into the diagram
\begin{equation*}
\begin{tikzcd}
\Fun([1],\C_0)\arrow[d,"d^{\ast}"]\arrow[r,hookrightarrow,"\mathrm{fact}"] & \Fun([2],\C_0)\arrow[d,"d^{\ast}"]\\
\Fun([1],\C_{-1})\arrow[r,hookrightarrow,"\mathrm{fact}"] & \Fun([2],\C_{-1}).
\end{tikzcd}
\end{equation*}
The vertical functors are essentially surjective (since they admit sections), so we obtain an induced factorization system $\C_{-1}^{\dagger}$ whose egressive resp. ingressive morphisms are those equivalent to images of egressive resp. ingressive morphisms under $d^{\ast}\colon\C_0\rightarrow\C_{-1}$. Thus, we have lifted the augmented simplicial object $\C_{\bullet}\colon\Delta_+^{\op}\rightarrow\Cat$ through $p\colon\Fact\rightarrow\Cat$ and need to prove that this is also a colimit diagram in $\Fact$. To this end, let $\D^{\dagger}$ be a category equipped with a factorization system and consider the commutative square
\begin{equation*}
\begin{tikzcd}
\Map_{\Fact}(\C_{-1}^{\dagger},\D^{\dagger})\arrow[r]\arrow[d,hookrightarrow] & \lim_{[n]\in\Delta}\Map_{\Fact}(\C_n^{\dagger},\D^{\dagger})\arrow[d,hookrightarrow]\\
\Map_{\Cat}(\C_{-1},\D)\arrow[r,"\sim"] & \lim_{[n]\in\Delta}\Map_{\Cat}(\C_n,\D).
\end{tikzcd}
\end{equation*}
The vertical maps are monomorphisms, so the upper horizontal map is a monomorphism by left cancellation and it remains to prove that it is essentially surjective. This, however, follows since a functor $f\colon\C_{-1}\rightarrow\D$ preserves egressive resp. ingressive morphisms if and only if the composite $fd^{\ast}\colon\C_0\rightarrow\C_{-1}\rightarrow\D$ preserves egressive resp. ingressive morphisms by the preceding description of $\C_{-1}^{\dagger}$.
\end{proof}


\printbibliography

@book{HTT,
    author = {Lurie, Jacob},
    year = {2009},
    title = {Higher Topos Theory},
    series = {Annals of Mathematics Studies},
    number = {170},
    publisher = {Princeton University Press}
}

@misc{kerodon,
    title = {Kerodon},
    author = {Lurie, Jacob},
    howpublished = {\url{https://kerodon.net}},
    year = {2018}
}

@misc{HA,
    title = {Higher Algebra},
    author = {Lurie, Jacob},
    howpublished = {\url{https://www.math.ias.edu/~lurie/papers/HA.pdf}},
    year = {2017}
}

@misc{BarkanSteinebrunner,
    title = {On the definition of factorization systems},
    author = {Barkan, Shaul and Steinebrunner, Jan},
    year = {2025},
    howpublished = {\url{https://www.jan-steinebrunner.com/wp-content/uploads/2025/12/On_the_definition_of_factorization_systems.pdf}}
}

@article{ChuHaugseng,
    title = {Homotopy-coherent algebra via Segal conditions},
    journal = {Advances in Mathematics},
    volume = {385},
    eid = {107733},
    year = {2021},
    author = {Chu, Hongyi and Haugseng, Rune},
}

@misc{Shah,
    title = {Parametrized higher category theory II: Universal constructions}, 
    author = {Shah, Jay},
    year = {2022},
    eprint = {2109.11954},
    archivePrefix = {arXiv} 
}

@article{HHNL,
    title = {Two-variable fibrations, factorisation systems and $\infty $-categories of spans},
    volume = {11},
    journal = {Forum of Mathematics, Sigma},
    author = {Haugseng, Rune and Hebestreit, Fabian and Linskens, Sil and Nuiten, Joost},
    year = {2023},
    eid = {e111}
}

@misc{Juran,
    title = {On orthogonal factorization systems and double categories}, 
    author = {Juran, Branko},
    year = {2025},
    eprint = {2501.01363},
    archivePrefix = {arXiv}
}

@misc{Fernandes,
    title = {Stable moduli spaces of odd-dimensional manifold triads}, 
    author = {Fernandes, João Lobo},
    year = {2025},
    eprint = {2510.17986},
    archivePrefix = {arXiv}
}

@misc{Haugseng,
    title = {Yet another introduction to $\infty$-Categories},
    author = {Haugseng, Rune},
    year = {2025},
    howpublished = {\url{https://runegha.folk.ntnu.no/naivecat_web.pdf}}
}

@misc{Lurie,
    author = {Lurie, Jacob},
    title = {$(\infty,2)$-Categories and the Goodwillie Calculus I}, 
    year = {2009},
    eprint = {0905.0462},
    archivePrefix = {arXiv}
}

@article{ABFJ,
    title = {Left-exact localizations of $\infty$-topoi I: Higher sheaves},
    journal = {Advances in Mathematics},
    volume = {400},
    eid = {108268},
    year = {2022},
    author = {Anel, Mathieu and Biedermann, Georg and Finster, Eric and Joyal, André}
}

@misc{Soergel,
    title = {A braided monoidal $(\infty,2)$-category of Soergel bimodules}, 
    author = {Liu, Yu Leon and Mazel-Gee, Aaron and Reutter, David and Stroppel, Catharina and Wedrich, Paul},
    year = {2024},
    eprint = {2401.02956},
    archivePrefix = {arXiv}
}

@misc{Kositsyn,
    title = {Factorization systems in $\infty$-categories}, 
    author = {Kositsyn, Roman},
    year = {2021},
    eprint = {2105.14654},
    archivePrefix = {arXiv}
}

@article{KorostenskiTholen,
    title = {Factorization systems as Eilenberg-Moore algebras},
    journal = {Journal of Pure and Applied Algebra},
    volume = {85},
    number = {1},
    pages = {57-72},
    year = {1993},
    author = {Korostenski, Mareli and Tholen, Walter}
}

@article{BarwickSchommerPries,
    title = {On the unicity of the theory of higher categories},
    author = {Barwick, Clark and Schommer-Pries, Christopher},
    journal = {Journal of the American Mathematical Society},
    volume = {34},
    number = {4},
    pages = {1011--1058},
    year = {2021}
}

@article{JohnsonFreydScheimbauer,
    title = {(Op)lax natural transformations, twisted quantum field theories, and “even higher” Morita categories},
    journal = {Advances in Mathematics},
    volume = {307},
    pages = {147-223},
    year = {2017},
    author = {Johnson-Freyd, Theo and Scheimbauer, Claudia}
}

@inproceedings{JoyalTierney,
    title = {Quasi-categories vs Segal spaces},
    author = {Joyal, Andr{\'e} and Tierney, Myles},
    series = {Contemporary Mathematics},
    booktitle = {Categories in Algebra, Geometry and Mathematical Physics},
    volume = {431},
    pages = {277-326},
    year = {2007},
    publisher = {American Mathematical Society}
}

@article{HebestreitSteinebrunner,
    author = {Hebestreit, Fabian and Steinebrunner, Jan},
    title = {A Short Proof That Rezk’s Nerve Is Fully Faithful},
    journal = {International Mathematics Research Notices},
    volume = {2025},
    number = {4},
    year = {2025},
}

@misc{LoubatonRuit,
    title = {On the squares functor and the Gaitsgory-Rozenblyum conjectures}, 
    author = {Loubaton, Félix and Ruit, Jaco},
    year = {2025},
    eprint = {2507.07807},
    archivePrefix = {arXiv}
}

@article{Ramzi,
    title = {An elementary proof of the naturality of the Yoneda embedding},
    author = {Ramzi, Maxime},
    journal = {Proceedings of the American Mathematical Society},
    volume = {151},
    number = {10},
    pages = {4163--4171},
    year = {2023}
}

@book{Hu,
    title = {Homotopy Theory},
    author = {Hu, Sze-Tsen},
    year = {1959},
    publisher = {Academic Press}
}

@misc{Joyal,
    title = {Notes on Logoi},
    author = {Joyal, André},
    year = {2008},
    howpublished = {\url{https://www.math.uchicago.edu/~may/IMA/JOYAL/Joyal.pdf}}
}

@misc{ORW,
    title={Cores and localizations of $(\infty,\infty)$-categories}, 
    author={Ozornova, Viktoriya and Rovelli, Martina and Walde, Tashi},
    year={2026},
    eprint={2603.11005},
    archivePrefix={arXiv}
}

@misc{AGH,
    title={Straightening for lax transformations and adjunctions of $(\infty,2)$-categories}, 
    author={Abellán, Fernando and Gagna, Andrea and Haugseng, Rune},
    year={2024},
    eprint={2404.03971},
    archivePrefix={arXiv}
}

@misc{AHM,
    title = {Free fibrations, lax colimits and Kan extensions for $(\infty,2)$-categories}, 
    author = {Abellán, Fernando and Haugseng, Rune and Martini, Louis},
    year = {2026},
    eprint = {2602.07604},
    archivePrefix = {arXiv}
}

@article{Segalification,
    title = {Segalification and the Boardman--Vogt tensor product},
    author = {Barkan, Shaul and Steinebrunner, Jan},
    journal = {Algebraic \& Geometric Topology},
    volume = {25},
    number = {9},
    pages = {5439--5462},
    year = {2025},
    publisher = {Mathematical Sciences Publishers}
}

@misc{CampionMaehara,
    title = {A model-independent Gray tensor product for $(\infty,2)$-categories}, 
    author = {Campion, Timothy and Maehara, Yuki},
    year = {2023},
    eprint = {2304.05965},
    archivePrefix = {arXiv}
}
\end{document}